\documentclass[12pt,leqno,a4paper]{article}

\usepackage[english]{babel}
\usepackage{amsmath}
\usepackage{amsfonts}
\usepackage{amsthm}
\usepackage{amssymb}

\usepackage{xspace}
\usepackage{euscript}
\usepackage{graphicx}
\usepackage{amscd}
\usepackage{tabularx}

\usepackage{enumerate}
 \usepackage{epsfig} 
 \usepackage{graphics} 
\theoremstyle{plain}
\newtheorem{theo}{Theorem}[section]
\newtheorem{lem}[theo]{Lemma}
\newtheorem{prop}[theo]{Proposition}
\newtheorem{cor}[theo]{Corollary}%

\theoremstyle{definition}
\newtheorem{definition}[theo]{Definition}
\theoremstyle{remark}
\newtheorem{rem}[theo]{Remark}
\numberwithin{equation}{section}

\newcommand{\C}{\mathbb{C}}

\newcommand{\R}{\mathbb{R}}
\newcommand{\N}{\mathbb{N}}
\newcommand{\M}{\mathbb{M}}
\newcommand{\divrg}{\textrm{div}\,}

\title{Sharp three sphere inequality for perturbations of a product of two second order elliptic
operators and stability for the Cauchy problem for the anisotropic plate equation
\thanks{Work supported by PRIN No. 20089PWTPS}}
\author{Antonino Morassi\thanks{Dipartimento di Georisorse e Territorio,
Universit\`a degli Studi di Udine, via Cotonificio 114, 33100
Udine, Italy. E-mail: \textsf{antonino.morassi@uniud.it}}, \  Edi
Rosset\thanks{Dipartimento di Matematica e Informatica,
Universit\`a degli Studi di Trieste, via Valerio 12/1, 34127
Trieste, Italy. E-mail: \textsf{rossedi@univ.trieste.it}} \ and
Sergio Vessella\thanks{DIMAD, Universit\`a degli Studi di Firenze,
Via Lombroso 6/17, 50134 Firenze, Italy,
\textsf{sergio.vessella@dmd.unifi.it}}}

\begin{document}

\maketitle

\begin{abstract}
We prove a sharp three sphere inequality for solutions to third
order perturbations of a product of two second order elliptic
operators with real coefficients. Then we derive various kinds of
quantitative estimates of unique continuation for the anisotropic
plate equation. Among these, we prove a stability estimate for the
Cauchy problem for such an equation and we illustrate some
applications to the size estimates of an unknown inclusion made of
different material that might be present in the plate. The paper is
self-contained and the Carleman estimate, from which the sharp three
sphere inequality is derived, is proved in an elementary and direct
way based on standard integration by parts.

\end{abstract}

\centerline{}

\section{Introduction}
 In the present paper we shall prove some
quantitative estimates of unique continuation for fourth order
elliptic equations arising in linear elasticity theory.

 The equations we are most concerned with are those
 describing the equilibrium of a thin plate having uniform
 thickness. Working in the framework of the linear elasticity for
 infinitesimal deformations and under the kinematical assumptions
 of the Kirchhoff-Love theory (see \cite{Fi}, \cite{Gu}), the transversal
 displacement $u$ of the plate satisfies the following equation

\begin{equation}
\label{1-I} \mathcal{L}u:=\sum_{i,j,k,l=1}^2\partial_{ij}^2
(C_{ijkl}(x)\partial_{kl}^2 u)=0, \quad \hbox{in } \Omega,
\end{equation}
where $\Omega$ is the middle surface of the plate and
$\{C_{ijkl}(x)\}_{i,j,k,l=1}^{2}$ is a fourth order tensor
describing the response of the material of the plate. In the
sequel we shall assume that the following standard symmetry
conditions are satisfied
\begin{equation}
    \label{2-I}
    C_{ijkl}(x)=C_{klij}(x)=C_{lkij}(x), \quad \hbox{ }{i,j,k,l=1,2},
    \quad\hbox{  in } \Omega.
\end{equation}
In addition we shall assume that $C_{ijkl}\in
C^{1,1}(\overline{\Omega})$, $i,j,k,l=1,2$, and that the following
strong convexity condition is satisfied
\begin{equation}
    \label{3-I}
    C_{ijkl}(x)A_{ij}A_{kl} \geq \gamma |A|^2,
    \quad \hbox{  in } \Omega,
\end{equation}
for every $2\times2$ symmetric matrix $A=\{A_{ij}\}_{i,j=1}^2$,
where $\gamma$ is a positive constant and
$|A|^2=\sum_{i,j=1}^2A_{ij}^2$.

More precisely, the quantitative estimates of unique continuation
which we obtain are in the form of a three sphere inequality (see
Theorem \ref{theo:9-4.3}, Theorem \ref{theo:12-4.3} and Theorem
\ref{theo:13-4.3}), in developing which we have mainly had in mind
its applications to two kinds of inverse problems for thin elastic
plates:

a) the stability issue for the inverse problem of the
determination of unknown boundaries,

b) the derivation of size estimates for unknown inclusions made of different
elastic material.\\

Let us give a brief description of problems a) and b).\\

\textit{Problem a)}. We consider a thin elastic plate, having
middle surface $\Omega$, whose boundary is made by an accessible
portion $\Gamma$ and by an unknown inaccessible portion $I$, to be
determined. Assuming that the boundary portion $I$ is free, a
possible approach to determine $I$ consists in applying a couple
field $\widehat{M}$ on $\Gamma$ and measuring the resulting transversal
displacement $u$ and its normal derivative $\frac{\partial
u}{\partial n}$ on an open subset of $\Gamma$. In \cite{M-R} it
was proved that, under suitable a priori assumptions, a single
measurement of this kind is sufficient to detect $I$. The
stability issue, which we address here, asks whether small
perturbations of the measurements produce or not small
perturbations of the unknown boundary $I$. Since assigning a
couple field $\widehat{M}$ results in prescribing the so called Neumann
conditions for the plate, that is two boundary conditions of
second and third order respectively, it follows that Cauchy data
are known in $\Gamma$. Therefore it is quite reasonable, also in
view of the literature about stability results for the
determination of unknown boundaries in other physical frameworks
(see for instance \cite{A-B-R-V}, \cite{Si}, \cite{Ve}), that the
first step to be proved in order to get such a
stability result consists in stability estimates for the Cauchy problem
for the fourth order equation \eqref{1-I}. For this reason, in the
present paper we derive a stability result for the Cauchy
problem, see Theorem \ref{theo:LSC}, having in mind applications to this
inverse problem and to the analogous ones, consisting in the
determination of cavities or rigid inclusions inside the plate. We
refer to \cite{M-R-V3} and to \cite{M-R} respectively for
uniqueness results for these two
inverse problems.\\

\textit{Problem b)}. We consider a thin elastic plate, inside which
an unknown inclusion made of different material might be present.
Denoting by $\Omega$ and $D$ the middle surface of the plate and of
the inclusion respectively, a problem of practical interest is the
evaluation of the area of $D$. In \cite{M-R-V1} we derived upper and
lower estimates of the area of $D$ in terms of boundary
measurements, for the case of isotropic material and assuming a
``fatness'' condition on the set $D$, see \cite[Theorem 4.1]{M-R-V1}.
Since the proof of that result was mainly based on a three sphere
inequality for $|\nabla^2u|^2$ (here $\nabla^2u$ denotes the Hessian matrix of $u$),
where $u$ is a solution of the plate
equation, we emphasize here that Theorem 4.1 of \cite{M-R-V1}
extends to the more general anisotropic assumptions on the
elasticity tensor stated in Theorem \ref{theo:12-4.3} of the present
paper, in which such a three sphere inequality is established.\\

Concerning the Cauchy problem, along a classical path, \cite{NIR},
recently revived in \cite{A-R-R-V} in the framework of second
order elliptic equations, we derive the stability estimates for
the Cauchy problem for equation \eqref{1-I} as a consequence of
smallness propagation estimates from an open set for solution to
\eqref{1-I}. Such smallness propagation estimates are achieved by
a standard iterative application of the three sphere inequality.

In view of the applications to problems a) and
b), we took care to study with particular attention the sharp
character of the exponents appearing in the three sphere
inequality because of its natural connection with the unique
continuation property for functions vanishing at a point with
polynomial rate of convergence (strong unique continuation property, \cite{CoGr},
\cite{Co-Gr-T}, \cite{Ge}, \cite {LeB}, \cite{L-N-W},
\cite{M-R-V1}) or with exponential rate of convergence, \cite{Co-K}, \cite{Pr}.
As a byproduct of our three sphere
inequality, we reobtain the result in \cite{Co-K}, in the case of
$C^{1,1}$ coefficients, stating that, if
$u(x)=O\left(e^{-|x-x_0|^{-\beta}}\right)$ as $x\rightarrow x_0$,
for some $x_0\in\Omega$ and for an appropriate
$\beta>0$ which is precisely defined below, then $u\equiv0$ in $\Omega$. Indeed it is not
worthless to stress that such kinds of unique continuation
properties, especially the quantitative version of the strong
unique continuation property (three sphere inequalities with
optimal exponent and doubling inequalities, in the interior and at
the boundary) have provided crucial tools to prove optimal
stability estimates for inverse problems with unknown boundaries
\cite{A-B-R-V}, \cite{Si}, \cite{Ve} and to get size estimates for
unknown inclusions, \cite{A-M-R1}, \cite{A-M-R2}, \cite{A-M-R3},
\cite{A-R-S}, \cite{M-R-V1}, \cite{M-R-V2}. Concerning
problem b), we stress that the application of doubling inequalities
allows to get size estimates of the unknown inclusion $D$ under fully
general hypotheses on $D$, which is assumed to be merely a measurable set, see
\cite{M-R-V2}.

The strong unique continuation property for equation \eqref{1-I}
holds true, \cite{CoGr}, \cite {LeB}, \cite{L-N-W},
\cite{M-R-V1}), when the tensor $\{C_{ijkl}(x)\}_{i,j,k,l=1}^{2}$
satisfies isotropy hypotheses, that is
\begin{equation}
    \label{10-I}
    C_{ijkl}(x)=\delta_{ij}\delta_{kl}\lambda(x)+\left(\delta_{ik}\delta_{jl}
    +\delta_{il}\delta_{jk}\right)\mu(x), \quad \hbox{ }{i,j,k,l=1,2},
    \quad \hbox{  in } \Omega,
\end{equation}

\noindent
where $\lambda$ and $\mu$ are the Lam\'{e} moduli.

On the other hand, in view of Alinhac Theorem \cite{Ali}, it seems
extremely improbable that the solutions to \eqref{1-I} can satisfy
the strong unique continuation property under the general
hypotheses \eqref{2-I} and \eqref{3-I} on the tensor $\{C_{ijkl}(x)\}_{i,j,k,l=1}^{2}$.
Indeed, let
$\widetilde{\mathcal{L}}=\sum_{h=0}^{4}a_{4-h}(x)\partial
_{1}^{h}\partial _{2}^{4-h}$ be the principal part of the operator
$\mathcal{L}$. Let $z_1, z_2, \overline{z}_1, \overline{z}_2$
(here $\overline{z}_j$ is the conjugate of the complex number
$z_j$) be the complex roots of the algebraic equation
$\sum_{h=0}^{4}a_{4-h}(x_0)z^{h}=0$. In \cite{Ali} it is proved
that if $z_1\neq z_2$ then there exists an operator $Q$ of order
less than four such that the strong unique continuation property
in $x_0$ doesn't hold true for the solutions to the equation
$\widetilde{\mathcal{L}}u+Qu=0$. A fortiori, it seems hopeless the
possibility that solutions to \eqref{1-I} can satisfy
the doubling inequality.

At the best of our knowledge, concerning both weak and strong
unique continuation property for equation \eqref{1-I}, under the
general assumptions \eqref{2-I}, \eqref{3-I} and some reasonable
smoothing condition on the coefficients $C_{ijkl}$, neither positive answers nor
counterexamples are available in the literature. On the other
hand, it is clear that, in order to face the issue of unique
continuation property for equation \eqref{1-I} under the above
mentioned conditions, the two-dimensional character of equation
\eqref{1-I} or the specific structure of the equation should play
a crucial role. Indeed, a Pl\u{\i}s's example, \cite{Pl},
\cite{Zu}, shows that the unique continuation property fails for
general three-dimensional fourth order elliptic equations with real
$C^\infty$ coefficients.

For the reasons we have just outlined, in the present paper we
have a bit departed from the specific equation \eqref{1-I} and we
have derived the three sphere inequality that we are interested
in, as a consequence of a three sphere inequality for solutions to
the equation
\begin{equation}
\label{4-I} P_4(u)+Q(u)=0, \quad
    \hbox{  in } B_1=\{x\in \mathbb{R}^n\ |\ |x|<1\},
\end{equation}
where $n\geq 2$, $Q$ is a third order operator with bounded
coefficients and $P_4$ is a fourth order elliptic operator such
that
\begin{equation}
\label{5-I} P_4=L_2L_1,
\end{equation}
where $L_1$ and $L_2$ are two second order uniformly elliptic
operator with real and $C^{1,1}(\overline{B_1})$ coefficients. Our
approach is also supported by the fact that the operator
$\mathcal{L}$ can be written, under very general and simple
conditions (see sections \ref{SecCauchy} and \ref{Sec4.3}), as follows
\begin{equation}
\label{6-I} \mathcal{L}=P_4+Q,
\end{equation}
where $P_4$ satisfies \eqref{5-I} and $Q$ is a third order operator
with bounded coefficients. We
have conventionally labeled such conditions (see Definition \ref{def:dichotomy} in
Section \ref{SecCauchy}) the \emph{dichotomy condition}.
On the other hand, the conditions under
which the decomposition \eqref{6-I} is possible are, up to now,
basically the same under which the unique continuation property
holds for fourth order elliptic equation in two variables
\cite{Wat}, \cite{Zu}. More precisely, such conditions guarantee the
weak unique continuation property for solution to $\mathcal{L}u=0$
provided that the complex characteristic lines of the principal part
of operator $\mathcal{L}$ satisfy some regularity hypothesis.

We prove the three sphere inequality for solutions to equation
\eqref{4-I} (provided that $P_4$ satisfies \eqref{5-I}) in Theorem
\ref{theo:7-4.2}. By such a theorem we immediately deduce,
Corollary \ref{cor:8-4.2}, the following unique continuation
property. Let $L_k=\sum_{i,j=1}^ng_k^{ij}(x)\partial_{ij}^2$,
$k=1,2$, where $g_k=\{g_k^{ij}(x)\}_{i,j=1}^{n}$ are symmetric
valued function whose entries belong to
$C^{1,1}\left(\overline{B}_1\right)$. Assuming that
$\{g_k^{ij}(x)\}_{i,j=1}^{n}$, $k=1,2$ satisfy a uniform
ellipticity condition in $B_1$, let $\nu_*$ and $\nu^*$ ($\mu_*$
and $\mu^*$) be the minimum and the maximum eigenvalues of
$\{g_1^{ij}(0)\}_{i,j=1}^{n}$ ($\{g_2^{ij}(0)\}_{i,j=1}^{n}$)
respectively, and let
$\beta>\sqrt{\frac{\mu^*\nu^*}{\mu_*\nu_*}}-1$. We have that
\begin{equation}
\label{30-I} \quad
    \hbox{if} \qquad
    u(x)=O\left(e^{-|x|^{-\beta}}\right), \quad
    \hbox{  as } x\rightarrow 0, \quad
    \hbox{  then } u\equiv0 \quad
    \hbox{  in } B_1.
\end{equation}

Since \eqref{30-I} has been proved for the first time in
\cite{Co-K}, see also \cite{Co-Gr-T}, where the sharp character of
property \eqref{30-I} has been emphasized,  we believe useful to
compare our procedure with the one followed in \cite{Co-K}. In the
present paper, as well as in \cite{Co-K}, the bulk of the proof
consists in obtaining a Carleman estimate for $P_4=L_2L_1$ with
weight function $e^{-\left(\sigma_0(x)\right)^{-\beta}}$, where
$\beta>\sqrt{\frac{\mu^*\nu^*}{\mu_*\nu_*}}-1$ and
$\left(\sigma_0(x)\right)^2$ is a suitable positive definite
quadratic form (Theorem \ref{theo:6-4.2}). In turn, here and in
\cite{Co-K}, the Carleman estimate for $P_4$ is obtained by an
iteration of two Carleman estimates for the operators $L_1$ and
$L_2$ with the same weight function
$e^{-\left(\sigma_0(x)\right)^{-\beta}}$. However, while in
\cite{Co-K} and \cite{Co-Gr-T} the proof of Carleman estimates for
$L_1$ and $L_2$ is carried out by a careful analysis of the
pseudoconvexity conditions, \cite{HO63}, \cite{HO2}, \cite{I04},
in the present paper, Section \ref{Sec4.1}, we obtain the same
estimates by a more elementary and direct way. More precisely, we
adapt appropriately a technique introduced in \cite{E-V} in the
context of parabolic operators. A prototype of this technique was
already used in \cite{Ke-Wa} in the  issue of the boundary unique
continuation for harmonic functions. Such a technique, which is
based only on integration by parts and on the fundamental theorem
of calculus, being direct and elementary, makes it possible to
easily control the constants that occur in the final three sphere
inequality.

Finally, let us notice that the above results can be extended also
to treat fourth order operators having leading part $\mathcal{L}u$
given by \eqref{1-I} and involving lower order terms. An example
of practical relevance is, for instance, the equilibrium problem
for a thin plate resting on an elastic foundation. According to
the Winkler model \cite{Win}, the corresponding equation is
\begin{equation}
\label{lower_order} \mathcal{L}u+ku=0, \quad \hbox{in } \Omega,
\end{equation}
where $k=k(x)$ is a smooth, strictly positive function. Indeed, 
in view of Theorem \ref{theo:7-4.2}, the three
sphere inequalities established in Section \ref{Sec4.3} extend to equation \eqref{lower_order}.

The plan of the paper is as follows. In Section \ref{SecNotation} we
introduce some basic notation. In Section \ref{SecCauchy} we present
the main results for the Cauchy problem, see Theorem \ref{theo:LSC}.
In Section \ref{Sec4.1} we prove a Carleman estimate for second
order elliptic operators, Theorem \ref{theo:4-4.1}, which will be
used in Section \ref{Sec4.2} to derive a Carleman estimate for
fourth order operators obtained as composition of two second order
elliptic operators, Theorem \ref{theo:6-4.2}. In the same Section,
as a consequence of Theorem \ref{theo:6-4.2}, we also derive a three
sphere inequality and the unique continuation property for such
fourth order operators, see Theorem \ref{theo:7-4.2} and Corollary
\ref{cor:8-4.2} respectively. Finally, in Section \ref{Sec4.3}, the results of Section
\ref{Sec4.2} are applied to the anisotropic plate operator,
obtaining the desired three sphere inequality, see Theorems
\ref{theo:9-4.3}, \ref{theo:12-4.3} and \ref{theo:13-4.3}.

\section{Notation\label{SecNotation}}

Let $P=(x_1(P), x_2(P))$ be a point of $\R^2$.
We shall denote by $B_r(P)$ the ball in $\R^2$ of radius $r$ and
center $P$ and by $R_{a,b}(P)$ the rectangle of center $P$ and sides parallel
to the coordinate axes, of length $a$ and $b$, namely
$R_{a,b}(P)=\{x=(x_1,x_2)\ |\ |x_1-x_1(P)|<a,\ |x_2-x_2(P)|<b \}$. To simplify the notation,
we shall denote $B_r=B_r(O)$, $R_{a,b}=R_{a,b}(O)$.

\noindent When representing locally a boundary as a graph, we use
the following definition.
\begin{definition}
  \label{def:2.1} (${C}^{k,\alpha}$ regularity)
Let $\Omega$ be a bounded domain in ${\R}^{2}$. Given $k,\alpha$,
with $k\in\N$, $0<\alpha\leq 1$, we say that a portion $S$ of
$\partial \Omega$ is of \textit{class ${C}^{k,\alpha}$ with
constants $\rho_{0}$, $M_{0}>0$}, if, for any $P \in S$, there
exists a rigid transformation of coordinates under which we have
$P=0$ and
\begin{equation*}
  \Omega \cap R_{\frac{\rho_0}{M_0},\rho_0}=\{x=(x_1,x_2) \in R_{\frac{\rho_0}{M_0},\rho_0}\quad | \quad
x_{2}>\psi(x_1)
  \},
\end{equation*}
where $\psi$ is a ${C}^{k,\alpha}$ function on
$\left(-\frac{\rho_0}{M_0},\frac{\rho_0}{M_0}\right)$ satisfying
\begin{equation*}
\psi(0)=0,
\end{equation*}
\begin{equation*}
\psi' (0)=0, \quad \hbox {when } k \geq 1,
\end{equation*}
\begin{equation*}
\|\psi\|_{{C}^{k,\alpha}\left(-\frac{\rho_0}{M_0},\frac{\rho_0}{M_0}\right)} \leq M_{0}\rho_{0}.
\end{equation*}

\medskip
\noindent When $k=0$, $\alpha=1$, we also say that $S$ is of
\textit{Lipschitz class with constants $\rho_{0}$, $M_{0}$}.
\end{definition}
\begin{rem}
  \label{rem:2.1}
  We use the convention to normalize all norms in such a way that their
  terms are dimensionally homogeneous with the $L^\infty$ norm and coincide with the
  standard definition when the dimensional parameter equals one.
  For instance, the norm appearing above is meant as follows
\begin{equation*}
  \|\psi\|_{{C}^{k,\alpha}\left(-\frac{\rho_0}{M_0},\frac{\rho_0}{M_0}\right)} =
  \sum_{i=0}^k \rho_0^i
  \|\psi^{(i)}\|_{{L}^{\infty}\left(-\frac{\rho_0}{M_0},\frac{\rho_0}{M_0}\right)}+
  \rho_0^{k+\alpha}|\psi^{(k)}|_{\alpha, \left(-\frac{\rho_0}{M_0},\frac{\rho_0}{M_0}\right)},
\end{equation*}
where
\begin{equation*}
|\psi^{(k)}|_{\alpha,\left(-\frac{\rho_0}{M_0},\frac{\rho_0}{M_0}\right)}= \sup_
{\overset{\scriptstyle x', \ y'\in
\left(-\frac{\rho_0}{M_0},\frac{\rho_0}{M_0}\right)}{\scriptstyle x'\neq y'}}
\frac{|\psi^{(k)}(x')-\psi^{(k)}(y')|} {|x'-y'|^\alpha}.
\end{equation*}

Similarly, denoting by $\nabla^i u$ the vector which components are the derivatives of order $i$ of the function $u$,
\begin{equation*}
  \|u\|_{{C}^{k,1}(\Omega)} =\sum_{i=0}^{k+1}
  {\rho_{0}}^{i}\|{\nabla}^{i} u\|_{{L}^{\infty}(\Omega)},
\end{equation*}
\begin{equation*}
\|u\|_{L^2(\Omega)}=\rho_0^{-1}\left(\int_\Omega
u^2\right) ^{\frac{1}{2}},
\end{equation*}
\begin{equation*}
\|u\|_{H^m(\Omega)}=\rho_0^{-1}\left(\sum_{i=0}^m
\rho_0^{2i}\int_\Omega|\nabla^i u|^2\right)^{\frac{1}{2}},
\end{equation*}
and so on for boundary and trace norms such as
$\|\cdot\|_{H^{\frac{1}{2}}(\partial\Omega)}$,
$\|\cdot\|_{H^{-\frac{1}{2}}(\partial\Omega)}$.

Notice also that, when $\Omega=B_{R}(0)$, then $\Omega$ satisfies
Definition \ref{def:2.1} with $\rho_{0}=R$, $M_0=2$ and therefore, for
instance,
\begin{equation*}
\|u\|_{H^m(B_R)}=R^{-1}\left(\sum_{i=0}^m
R^{2i}\int_{B_R}|\nabla^i u|^2\right)^{\frac{1}{2}},
\end{equation*}
\end{rem}

Given a bounded domain $\Omega$ in $\R^2$ such that $\partial
\Omega$ is of class $C^{k,\alpha}$, with $k\geq 1$, we consider as
positive the orientation of the boundary induced by the outer unit
normal $n$ in the following sense. Given a point
$P\in\partial\Omega$, let us denote by $\tau=\tau(P)$ the unit
tangent at the boundary in $P$ obtained by applying to $n$ a
counterclockwise rotation of angle $\frac{\pi}{2}$, that is
\begin{equation}
    \label{eq:2.tangent}
        \tau=e_3 \times n,
\end{equation}
where $\times$ denotes the vector product in $\R^3$, $\{e_1,
e_2\}$ is the canonical basis in $\R^2$ and $e_3=e_1 \times e_2$.

Given any connected component $\cal C$ of $\partial \Omega$ and
fixed a point $P\in\cal C$, let us define as positive the
orientation of $\cal C$ associated to an arclength
parametrization $\varphi(s)=(x_1(s), x_2(s))$, $s \in [0, l(\cal
C)]$, such that $\varphi(0)=P$ and
$\varphi'(s)=\tau(\varphi(s))$. Here $l(\cal C)$ denotes the
length of $\cal C$.

Throughout the paper, we denote by $\partial_i u$, $\partial_s u$, and $\partial_n u$
the derivatives of a function $u$ with respect to the $x_i$
variable, to the arclength $s$ and to the normal direction $n$,
respectively, and similarly for higher order derivatives.

We denote by $\mathbb{M}^2$ the space of $2 \times 2$ real valued
matrices and by ${\mathcal L} (X, Y)$ the space of bounded linear
operators between Banach spaces $X$ and $Y$.

For every $2 \times 2$ matrices $A$, $B$ and for every $\mathbb{L}
\in{\mathcal L} ({\mathbb{M}}^{2}, {\mathbb{M}}^{2})$, we use the
following notation:
\begin{equation}
  \label{eq:2.notation_1}
  ({\mathbb{L}}A)_{ij} = L_{ijkl}A_{kl},
\end{equation}
\begin{equation}
  \label{eq:2.notation_2}
  A \cdot B = A_{ij}B_{ij},
\end{equation}
\begin{equation}
  \label{eq:2.notation_3}
  |A|= (A \cdot A)^{\frac {1} {2}},
\end{equation}
\begin{equation}
  \label{eq:2.notation_3bis}
  A^{sym} =  \frac{1}{2} \left ( A + A^t \right ),
\end{equation}
where $A^t$ denotes the transpose of the matrix $A$.
Notice that here and in the sequel summation over repeated indexes
is implied.

\section{Stability estimates for the Cauchy problem\label{SecCauchy}}

Let us consider a thin plate
$\Omega\times[-\frac{h}{2},\frac{h}{2}]$ with middle surface
represented by a bounded domain $\Omega$ in $\R^2$ and having
uniform thickness $h$, $h<<\hbox{diam}(\Omega)$. Given a positive constant $M_1$, we assume that
\begin{equation}
  \label{eq:M_1}
  |\Omega|\leq M_1\rho_0^2.
\end{equation}
Let us assume that the plate is made of nonhomogeneous linear
elastic material with elasticity tensor $\C(x) \in{\cal L}
({\M}^{2}, {\M}^{2})$ and that body forces inside $\Omega$ are
absent. We denote by $\hat M$ a couple field acting on the
boundary $\partial\Omega$.

We shall assume throughout that the elasticity tensor $\C$ has cartesian components
$C_{ijkl}$ which satisfy the following conditions
\begin{equation}
  \label{eq:sym-conditions-C-components}
    C_{ijkl} = C_{ klij} =
    C_{ klji} \quad i,j,k,l
    =1,2, \hbox{ a.e. in } \Omega.
\end{equation}
We recall that the symmetry conditions
\eqref{eq:sym-conditions-C-components} are equivalent to
\begin{equation}
  \label{eq:sym-conditions-C-1}
  {\C}A={\C} {A}^{sym},
\end{equation}
\begin{equation}
  \label{eq:eq:sym-conditions-C-2}
  {\C}A \quad \hbox{is } symmetric,
\end{equation}
\begin{equation}
  \label{eq:eq:sym-conditions-C-3}
  {\C}A \cdot B= {\C}B \cdot A,
\end{equation}
for every $2 \times 2$ matrices $A$, $B$.

In order to simplify the presentation, we shall assume that the tensor
$\mathbb{C}$ is defined in all of $\R^2$.

On the elasticity tensor $\mathbb{C}$ we make the following
assumptions:

\medskip
{I)} \textit{Regularity}
\begin{equation}
  \label{eq:3.bound}
  \mathbb{C} \in C^{1,1}(\R^2,   {\mathcal L} ({\mathbb{M}}^{2},
  {\mathbb{M}}^{2})),
\end{equation}
with
\begin{equation}
  \label{eq:3.bound_quantit}
  \sum_{i,j,k,l=1}^2 \sum_{m=0}^2 \rho_0^m \|\nabla^m C_{ijkl}\|_{L^\infty(\R^2)} \leq
    M,
\end{equation}
where $M$ is a positive constant;

{II)} \textit{Ellipticity (strong convexity)} There exists $\gamma>0$
such that

\begin{equation}
  \label{eq:3.convex}
    {\mathbb{C}}A \cdot A \geq  \gamma |A|^2, \qquad \hbox{in } \R^2,
\end{equation}
for every $2\times 2$ symmetric matrix $A$.

 Condition \eqref{eq:sym-conditions-C-components} implies that instead of $16$ coefficients
we actually deal with $6$ coefficients and we denote
\begin{center}
\( {\displaystyle \left\{
\begin{array}{lr}
    C_{1111}=A_0, \ \ C_{1122}=C_{2211}=B_0,
         \vspace{0.12em}\\
    C_{1112}=C_{1121}=C_{1211}=C_{2111}=C_0,
        \vspace{0.12em}\\
    C_{2212}=C_{2221}=C_{1222}=C_{2122}=D_0,
        \vspace{0.12em}\\
    C_{1212}=C_{1221}=C_{2112}=C_{2121}=E_0,
        \vspace{0.12em}\\
    C_{2222}=F_0,
        \vspace{0.25em}\\
\end{array}
\right. } \) \vskip -3.0em
\begin{eqnarray}
\ & & \label{3.coeff6}
\end{eqnarray}
\end{center}
and
\begin{equation}
    \label{3.coeffsmall}
    a_0=A_0, \ a_1=4C_0, \ a_2=2B_0+4E_0, \ a_3=4D_0, \ a_4=F_0.
\end{equation}

Let $S(x)$ be the following $7\times 7$ matrix
\begin{equation}
    \label{3. S(x)}
    S(x) = {\left(
\begin{array}{ccccccc}
  a_0   & a_1   & a_2   & a_3   & a_4   & 0    &    0    \\
  0     & a_0   & a_1   & a_2   & a_3   & a_4  &    0    \\
  0     & 0     & a_0   & a_1   & a_2   & a_3  &    a_4  \\
  4a_0  & 3a_1  & 2a_2  & a_3   & 0     & 0    &    0    \\
  0     & 4a_0  & 3a_1  & 2a_2  & a_3   & 0    &    0    \\
  0     & 0     & 4a_0  & 3a_1  & 2a_2  & a_3  &    0    \\
  0     & 0     & 0     & 4a_0  & 3a_1  & 2a_2 &    a_3  \\
  \end{array}
\right)},
\end{equation}
and
\begin{equation}
    \label{3.D(x)}
    {\mathcal{D}}(x)= \frac{1}{a_0} |\det S(x)|.
\end{equation}
Let us introduce the fourth order \emph{plate tensor}
\begin{equation}
    \label{3.P}
    \mathbb{P}= \frac{h^3}{12} \mathbb{C}, \quad\hbox{in } \R^2.
\end{equation}
With this notation we may rewrite the plate equation \eqref{1-I}
in the equivalent compact form

\begin{equation}
    \label{3.compact_plate}
{\rm div}({\rm div} (
      {\mathbb P}\nabla^2 u))=0, \quad\hbox{in } \Omega,
\end{equation}
where the divergence of a second order tensor field $T(x)$ is
defined, as usual, by
\begin{equation*}
  (\divrg T(x))_i=\partial_j T_{ij}(x).
\end{equation*}
Our approach to the Cauchy problem leads us to consider the
following complete, inhomogeneous equation
\begin{equation}
    \label{3.compact_plate_inhom}
{\rm div}({\rm div} (
      {\mathbb P}\nabla^2 u))=f + {\rm div}F + {\rm div}({\rm div} \mathcal{F}), \quad\hbox{in } B_R,
\end{equation}
where $f\in L^2(\R^2)$, $F\in L^2(\R^2;\R^2)$, $\mathcal{F}\in
L^2(\R^2;\mathbb{M}^2)$ satisfy the bound
\begin{equation}
    \label{3.bound_inhom}
\|f\|_{L^2(\R^2)}+\frac{1}{\rho_0}\|F\|_{L^2(\R^2;\R^2)}+
\frac{1}{\rho_0^2}\|\mathcal{F}\|_{L^2(\R^2;\mathbb{M}^2)} \leq
\frac{\epsilon}{\rho_0^4},
\end{equation}
for a given $\epsilon>0$.

A weak solution to \eqref{3.compact_plate_inhom} is a function
$u\in H^2(B_R)$ satisfying
\begin{equation}
    \label{3.inhom_weak}
\int_{B_R}\mathbb{P}\nabla^2
u\cdot\nabla^2\varphi=\int_{B_R}f\varphi-\int_{B_R}F\cdot\nabla\varphi+
\int_{B_R}\mathcal{F}\cdot\nabla^2\varphi, \quad\hbox{for every
}\varphi\in H^2_0(B_R).
\end{equation}

In the sequel we shall use the following condition on the elasticity
tensor that we have conventionally labeled \emph{dichotomy condition}.

\begin{definition}
  \label{def:dichotomy}(\textbf{Dichotomy condition}) Let $\mathcal{O}$ be an open set of $\mathbb{R}^2$. We shall say that the tensor
$\mathbb{P}$ satisfies the \emph{dichotomy condition} in $\mathcal{O}$ if
one of the following conditions holds true
\begin{subequations}
\begin{eqnarray}
\label{3.D(x)bound} &&  {\mathcal{D}}(x)>0,
\quad\hbox{for every } x\in \overline{\mathcal{O}}, \\[2mm]
\label{3.D(x)bound 2} && {\mathcal{D}}(x)=0, \quad\hbox{for every }
x\in \overline{\mathcal{O}},
\end{eqnarray}
\end{subequations}
where ${\mathcal{D}}(x)$ is defined by \eqref{3.D(x)}.
\end{definition}

\begin{rem}
  \label{rem:dichotomy} Whenever \eqref{3.D(x)bound} holds we denote
\begin{equation}
    \label{delta-1}
    \delta_1=\min_{\overline{\mathcal{O}}}{\mathcal{D}}.
\end{equation}
We emphasize that, in all the following statements, whenever a
constant is said to depend on $\delta_1$ (among other quantities) it
is understood that such dependence occurs \textit{only} when
\eqref{3.D(x)bound} holds.
\end{rem}

\begin{rem}
  \label{rem:orthotropy}
Let us briefly comment the \emph{dichotomy condition} in the
special class of \emph{orthotropic} materials, frequently used in
practical applications. In particular, let us assume that through
each point of the plate there pass three mutually orthogonal
planes of elastic symmetry and that these planes are parallel at
all points. In this case
\begin{equation}
    \label{ortho-1}
    C_0=0, \quad D_0=0,
\end{equation}
so that
\begin{equation}
    \label{ortho-2}
    a_0=A_0, \quad a_1=0, \quad a_2=2B_0+4E_0, \quad a_3=0, \quad
    a_4=F_0,
\end{equation}
and
\begin{equation}
    \label{ortho-3}
    {\mathcal{D}}(x) = 16 a_0 a_4 ( a_2^2 - 4 a_0 a_4)^2.
\end{equation}
Since, by the ellipticity condition \eqref{eq:3.convex}, the
coefficients $a_0$, $a_4$ are strictly positive, the dichotomy
condition reduces to the vanishing or not vanishing of the factor
$a_2^2 - 4 a_0 a_4$.

Introducing the engineering constitutive coefficients $E_1$,
$E_2$, $G_{12}$, $\nu_{12}$, $\nu_{21}$, with $\nu_{12} E_2 =
\nu_{21} E_1$ by the symmetry of $\mathbb{C}$, we have
\begin{equation}
    \label{ortho-4}
    a_2^2 - 4 a_0 a_4 = 4E_1^2
    \left (
    \left (
    \frac{\nu_{12}}{k} + \frac{1- \frac{\nu_{12}^2}{k}  }{ m+\nu_{12}}
    \right )^2
    -
    \frac{1}{k}
    \right ),
\end{equation}
where
\begin{equation}
    \label{ortho-5}
    k= \frac{E_1}{E_2}, \quad m=\frac{E_1}{2G_{12}}-\nu_{12}.
\end{equation}
The \emph{isotropic} case corresponds to $k=1$ and $m=1$, so that,
by \eqref{ortho-4}, ${\mathcal{D}}(x) \equiv 0$.

Let us notice that
\begin{equation}
    \label{ortho-6}
    \hbox{if } m = \sqrt k, \quad \hbox{then } {\mathcal{D}}(x) \equiv
0.
\end{equation}
This shows that there exist anisotropic materials such that
\eqref{3.D(x)bound 2} is satisfied. Roughly speaking, this simple
example makes clear that the value of ${\mathcal{D}}(x)$ cannot be
interpreted as a ``measure of anisotropy''.

Moreover, a case of practical interest corresponds to the
vanishing of the Poisson's coefficient $\nu_{12}$, which gives
\begin{equation}
    \label{ortho-7}
    a_2^2 - 4 a_0 a_4 = 4E_1^2
    \left (
    \frac{1}{m^2}-\frac{1}{k}
    \right ),
\end{equation}
so that
\begin{equation}
    \label{ortho-8}
    \hbox{if } m \neq \sqrt k, \quad \hbox{then } {\mathcal{D}}(x)
    > 0.
\end{equation}
This gives an explicit class of examples in which
\eqref{3.D(x)bound} holds.
\end{rem}

\begin{theo} [Three sphere inequality - complete equation]
   \label{theo:3sfere_completa}
Let $u \in H^4({B}_R)$ be a solution to the equation
\eqref{3.compact_plate_inhom}, where $\mathbb{P}$, defined by
\eqref{3.P}, satisfies \eqref{eq:sym-conditions-C-components},
\eqref{eq:3.bound_quantit}, \eqref{eq:3.convex} and the dichotomy condition
in $B_R$. There exist positive constants $k$ and $s$,
$k\in(0,1)$ only depending on $\gamma$ and
$M$, $s\in(0,1)$ only depending on
$\gamma$, $M$ and on
$\delta_1=\min_{\overline{B}_R}{\mathcal{D}}$, such that for every
$r_1$, $r_2$, $r_3$, $0<r_1<r_2<kr_3<sR$, the following inequality
holds
\begin{equation}
    \label{3.3sfere_eqcompl}
    \|u\|_{L^2(B_{r_2})}\leq C\left(\|u\|_{L^2(B_{r_1})}+\epsilon\right)^\alpha
    \left(\|u\|_{H^4(B_{r_3})}+\epsilon\right)^{1-\alpha}
\end{equation}
where $C>0$ and $\alpha\in(0,1)$ only depend on
$\gamma$, $M$, $\delta_1$, $\frac{r_2}{r_1}$,
$\frac{r_3}{r_2}$ and $\delta_1=\min_{\overline{B}_R}{\mathcal{D}}$.
\end{theo}
\begin{proof}
Let us consider the unique solution $u_0$ to
\begin{equation}
   \label{3.u_0}
\left\{
\begin{array}{lr}
     {\rm div}({\rm div} (
      {\mathbb P}\nabla^2 u_0))=f + {\rm div}F + {\rm div}({\rm div} \mathcal{F}), &\hbox{in } B_R,\\
      u_0=0,& \hbox{on }\partial B_R,\\
      \frac{\partial u_0}{\partial\nu}=0, &\hbox{on } \partial B_R.\\
\end{array}
\right.
\end{equation}
By using the weak formulation \eqref{3.inhom_weak} with
$\varphi=u_0$, by the strong convexity condition
\eqref{eq:3.convex}, by using the bound \eqref{3.bound_inhom} on
the inhomogeneous term and by Poincar\'{e} inequality in
$H^2_0(B_R)$, we have
\begin{equation}
    \label{3.u_0bound}
    \|u_0\|_{L^2(B_{R})}\leq \|u_0\|_{H^2_0(B_{R})}\leq C\epsilon,
\end{equation}
with $C$ only depending on $\gamma$.

Noticing that $u-u_0$ satisfies the hypotheses of Theorem
\ref{theo:13-4.3}, we have that the thesis immediately follows.
\end{proof}

Let $\Sigma$ be an open connected portion of $\partial \Omega$
such that $\Sigma$ is of class $C^{1,1}$ with constants $\rho_0$, $M_0$, and there exists a point $P_0 \in \Sigma$ such that
\begin{equation}
    \label{1-page6-par3}
    R_{ \frac{\rho_0}{M_0}, \rho_0}(P_0) \cap \partial \Omega
    \subset \Sigma.
\end{equation}
We shall consider as test function space the space
$H_{co}^2(\Omega \cup \Sigma)$ consisting of the functions
$\varphi \in H^2(\Omega)$ having support compactly contained in
$\Omega \cup \Sigma$. We denote by $H^{ \frac{3}{2}}_{co}(\Sigma)$
the class of $H^{ \frac{3}{2}}(\Sigma)$ traces of functions
$\varphi \in H_{co}^2(\Omega \cup \Sigma)$, and by $H^{
\frac{1}{2}}_{co}(\Sigma)$ the class of $H^{ \frac{1}{2}}(\Sigma)$
traces of the normal derivative $ \frac{\partial \varphi}{\partial
n}$ of functions $\varphi \in H_{co}^2(\Omega \cup \Sigma)$.
Moreover, for every positive integer number $m$, we define $H^{-
\frac{m}{2}}(\Sigma)$ as the dual space to
$H^{\frac{m}{2}}(\Sigma)$ based on the $L^2(\Sigma)$ dual pairing.
Let $g_1\in H^{ \frac{3}{2}}(\Sigma)$, $g_2\in H^{
\frac{1}{2}}(\Sigma)$ and $\widehat{M}\in H^{-
\frac{1}{2}}(\Sigma; \mathbb{R}^2)$ be such that
\begin{equation}
    \label{2-page6-par3}
    \|g_1\|_{H^{\frac{3}{2}}(\Sigma)}+\rho_0\|g_2\|_{H^{\frac{1}{2}}(\Sigma)}+
    \rho_0^2\|\widehat{M}\|_{H^{-\frac{1}{2}}(\Sigma; \R^2)} \leq \eta,
\end{equation}
for some positive constant $\eta$.

We consider the following Cauchy problem
\begin{center}
\( {\displaystyle \left\{
\begin{array}{lr}
  \divrg(\divrg (
  {\mathbb P}\nabla^2 u))=0,
  & \hbox{in}\ \Omega,
    \vspace{0.25em}\\
  u=g_1, & \hbox{on}\ \Sigma,
        \vspace{0.25em}\\
  \frac{\partial u}{\partial n}=g_2, & \hbox{on}\ \Sigma,
          \vspace{0.25em}\\
  ({\mathbb {P}}\nabla^2 u) n \cdot n=-\widehat{M}_n, & \hbox{on}\ \Sigma,
          \vspace{0.25em}\\
  \divrg({\mathbb {P}} \nabla^2 u)\cdot n + (({\mathbb {P}}\nabla^2 u) n \cdot \tau),_{s}=\widehat{M}_{\tau,s}, & \hbox{on}\
\Sigma,
\end{array}
\right. } \) \vskip -8.4em
\begin{eqnarray}
& & \label{eq:1-page7-par3}\\
& & \label{eq:2-page7-par3}\\
& & \label{eq:3-page7-par3}\\
& & \label{eq:4-page7-par3}\\
& & \label{eq:5-page7-par3}
\end{eqnarray}
\end{center}
where $\widehat{M}_\tau=\widehat{M}\cdot n$,
$\widehat{M}_n=\widehat{M}\cdot \tau$ denote respectively the
twisting moment and the bending moment applied at the boundary.

A weak solution to
\eqref{eq:1-page7-par3}--\eqref{eq:5-page7-par3} is a function $u
\in H^2(\Omega)$ such that
\begin{equation}
\label{6-page7-par3}
    \int_\Omega \mathbb P \nabla^2 u \cdot \nabla^2 \varphi =
    - \int_\Sigma \left (
    \widehat{M}_{\tau,s} \varphi + \widehat{M}_n \varphi_n \right
    ), \quad \hbox{for every } \varphi \in H_{co}^2(\Omega \cup
    \Sigma),
\end{equation}
with
\begin{equation}
\label{7-page7-par3}
    u|_{\Sigma} = g_1, \quad \frac{\partial u}{\partial n}|_{\Sigma}=g_2.
\end{equation}
We denote
\begin{equation}
\label{1-page8-par3}
    R^-_{ \frac{\rho_0}{M_0}, \rho_0} (P_0)=\{ (x_1,x_2) \in R_{
    \frac{\rho_0}{M_0}, \rho_0}(P_0) | \ x_2 < \psi(x_1)\},
\end{equation}
that is
\begin{equation}
\label{2-page8-par3}
    R^-_{ \frac{\rho_0}{M_0}, \rho_0} (P_0)=
    R_{\frac{\rho_0}{M_0}, \rho_0}(P_0) \setminus
    \overline{\Omega}.
\end{equation}
\begin{lem}
\label{lem:traces}
    Let $g_1 \in H^{ \frac{3}{2}}(\Sigma)$, $g_2 \in H^{
    \frac{1}{2}}(\Sigma)$. Then there exists $v \in H^2(R^-_{ \frac{\rho_0}{M_0}, \rho_0}
    (P_0))$ such that
\begin{equation}
\label{3-page8-par3}
    v|_{\Sigma \cap R_{\frac{\rho_0}{M_0}, \rho_0}(P_0)} = g_1,
\end{equation}
\begin{equation}
\label{4-page8-par3}
    \frac{\partial v}{\partial n}|_{\Sigma \cap R_{\frac{\rho_0}{M_0},
    \rho_0}(P_0)}= g_2
\end{equation}
and
\begin{equation}
\label{5-page8-par3}
    \|v\|_{H^2(R^-_{ \frac{\rho_0}{M_0}, \rho_0} (P_0))} \leq
    C \left ( \|g_1\|_{H^{\frac{3}{2}}(\Sigma)}+\rho_0  \|g_2\|_{H^{\frac{1}{2}}(\Sigma)}
    \right ),
\end{equation}
where $C$, $C>0$, only depends on $M_0$.
\end{lem}
\begin{proof} The proof follows the lines of the proof of Lemma
6.1 of \cite{A-R-R-V}.
\end{proof}
Let us define
\begin{equation}
    \label{1-page9-par3}
    \widetilde{u}= \left\{
    \begin{array}{ll}
        u, & \hbox{in } \Omega, \\
        & \\
        v & \hbox{in } R^-_{ \frac{\rho_0}{M_0}, \rho_0} (P_0),\\
    \end{array}
    \right.
\end{equation}
\begin{equation}
    \label{2-page9-par3}
    \Omega_1 = \Omega \cup \left ( \Sigma \cap R_{ \frac{\rho_0}{M_0}, \rho_0}
    (P_0) \right ) \cup R^-_{ \frac{\rho_0}{M_0}, \rho_0} (P_0).
\end{equation}
Since $u$ and $v$ share the same Dirichlet data $(g_1, g_2)$ on
$\Sigma$, we have that
\begin{equation}
    \label{3-page9-par3}
    \widetilde{u} \in H^2(\Omega_1).
\end{equation}
\begin{theo}
\label{theo:Extension}
There exist $\widetilde{f} \in L^2(\Omega_1)$, $\widetilde{F} \in
L^2(\Omega_1; \R^2)$, $\mathcal{F} \in L^2(\Omega_1;
\mathbb{M}^2)$ such that
\begin{equation}
    \label{4-page9-par3}
    \|\widetilde{f}\|_{L^2(\Omega_1)}+
    \frac{1}{\rho_0}\|\widetilde{F}\|_{L^2(\Omega_1;\R^2)} +
    \frac{1}{\rho_0^2}\|\mathcal{F}\|_{L^2(\Omega_1;
    \mathbb{M}^2)}\leq \frac{C\eta}{\rho_0^4}
\end{equation}
and $\widetilde{u}$ satisfies in the weak sense the equation
\begin{equation}
    \label{5-page9-par3}
    {\rm div}({\rm div} (
      {\mathbb P}\nabla^2 \widetilde{u}))=\widetilde{f} + {\rm div}\widetilde{F} + {\rm div}({\rm div} \mathcal{\widetilde{F}}), \quad\hbox{in }
      \Omega_1.
\end{equation}
Here, the constant $C$, $C>0$, only depends on $M_0$ and $\gamma$.
\end{theo}
\begin{proof}
Let $\varphi$ be an arbitrary test function in $H^2_0(\Omega_1)$.
It is clear that $\varphi|_{\Omega} \in H_{co}^2(\Omega \cup
\Sigma)$. Denoting for simplicity $R^- =R^-_{ \frac{\rho_0}{M_0},
\rho_0} (P_0)$, by \eqref{6-page7-par3} we have
\begin{equation}
    \label{1-page10-par3}
    \int_{\Omega_1} \mathbb{P} \nabla^2 \widetilde{u} \cdot
    \nabla^2 \varphi =
    - \int_\Sigma ( \widehat{M}_{\tau,s}\varphi +\widehat{M}_n
    \varphi_{,n}) + \int_{R^-} \mathbb{P} \nabla^2 v \cdot
    \nabla^2 \varphi.
\end{equation}
Let us define the functional $\Psi: H_0^2(\Omega_1) \rightarrow
\R$ as
\begin{equation}
    \label{2-page10-par3}
    \Psi(\varphi) = \int_\Sigma ( \widehat{M}_{\tau,s}\varphi +\widehat{M}_n
    \varphi_{,n}) = \rho_0 \left ( \frac{1}{\rho_0} \int_\Sigma ( \widehat{M}_{\tau,s}\varphi +\widehat{M}_n
    \varphi_{,n}) \right ).
\end{equation}
By standard trace embedding and by \eqref{2-page6-par3}, we have
\begin{multline}
    \label{1-page11-par3}
    |\Psi(\varphi)| \leq \rho_0 \left (
    \|\widehat{M}_{\tau,s}\|_{H^{- \frac{3}{2}}(\Sigma)}
    \|\varphi\|_{H^{\frac{3}{2}}(\Sigma)}+
    \|\widehat{M}_n\|_{H^{-\frac{1}{2}}(\Sigma)} \|\varphi_{,n}\|_{H^{\frac{1}{2}}(\Sigma)}
    \right )
    \leq
    \\
    \leq
    C \|\widehat{M}\|_{H^{-\frac{1}{2}}(\Sigma)} \|\varphi\|_{H_0^{2}(\Omega_1)}
    \leq
    \frac{C\eta}{\rho_0^2} \|\varphi\|_{H_0^{2}(\Omega_1)},
\end{multline}
where $C$, $C>0$, only depends on $M_0$.  Therefore, $\Psi \in
H^{-2}(\Omega_1)$ and
\begin{equation}
    \label{2-page11-par3}
    \|\Psi\|_{H^{-2}(\Omega_1)} \leq \frac{C\eta}{\rho_0^2}.
\end{equation}
By the well-known Riesz Representation Theorem in Hilbert spaces,
we can find $f \in H_0^2(\Omega_1)$ such that $\Psi(\varphi) =
<\varphi, f>_{H_0^2(\Omega_1)}$ for every $\varphi \in
{H_0^2(\Omega_1)}$ and
\begin{equation}
    \label{2BIS-page11-par3}
    \|\Psi\|_{H^{-2}(\Omega_1)} = \|f\|_{H_0^2(\Omega_1)}.
\end{equation}
Let us set
\begin{equation}
    \label{3-page11-par3}
    f_1 = \frac{f}{\rho_0^2}, \quad F_1= -\nabla f, \quad
    {\mathcal{F}_1} = \rho_0^2 \nabla^2 f.
\end{equation}
Then
\begin{equation}
    \label{4-page11-par3}
    \rho_0\|f_1\|_{L^2(\Omega_1)}+
    \|F_1\|_{L^2(\Omega_1; \R^2)} +
    \rho_0^{-1}\|\mathcal{F}_1\|_{L^2(\Omega_1;
    \mathbb{M}^2)}\leq \frac{C\eta}{\rho_0^3}.
\end{equation}
By \eqref{1-page10-par3}
\begin{equation}
    \label{1-page12-par3}
    \int_{\Omega_1} \mathbb{P} \nabla^2 \widetilde{u} \cdot
    \nabla^2 \varphi = \int_{R^-} \mathbb{P} \nabla^2 v \cdot
    \nabla^2 \varphi - \int_{\Omega_1}f_1\varphi +
    \int_{\Omega_1}F_1 \cdot \nabla \varphi - \int_{\Omega_1}
    \mathcal{F}_1 \cdot \nabla^2 \varphi,
\end{equation}
for every $\varphi \in H_0^2(\Omega_1)$. Denoting
\begin{equation}
    \label{2-page12-par3}
    \widetilde{f}=-f_1, \quad \widetilde{F}=-F_1, \quad
    \mathcal{ \widetilde{F} }  = \left\{
    \begin{array}{lr}
        - \mathcal{F}_1, & \hbox{in } \Omega_1, \\
        \mathbb{P}\nabla^2 v -  \mathcal{F}_1,  &\hbox{in } R^{-},
    \end{array}
    \right.
\end{equation}
we obtain \eqref{5-page9-par3}. By \eqref{2-page12-par3},
\eqref{3-page11-par3}, \eqref{eq:3.bound_quantit},
\eqref{2-page11-par3}, \eqref{2BIS-page11-par3},
\eqref{5-page8-par3}, \eqref{2-page6-par3} we obtain
\eqref{4-page9-par3}.
\end{proof}

\begin{theo} [Propagation of smallness in the interior]
\label{theo:HPS}
Let $\Omega$ be a bounded domain in $\R^2$ satisfying \eqref{eq:M_1} and let $B_{r_0}(x_0)\subset\Omega$
be a fixed disc. Let $r$, $0<r\leq \frac{r_0}{2}$ be fixed and let $G\subset\Omega$ be a connected
open set such that ${\rm dist}(G,\partial\Omega)\geq r$ and $B_{\frac{r_0}{2}}(x_0)\subset G$.
Let $u\in H^2_{loc}(\Omega)$ be  a weak solution to the equation
\begin{equation}
   \label{1-page13-par3}
{\rm div}({\rm div} (
      {\mathbb P}\nabla^2 u_0))=f + {\rm div}F + {\rm div}({\rm div} \mathcal{F}), \quad\hbox{in } \Omega
\end{equation}
where $\mathbb{P}$, defined by \eqref{3.P}, satisfies
\eqref{eq:sym-conditions-C-components},
   \eqref{eq:3.bound_quantit}, \eqref{eq:3.convex} and the dichotomy condition
in $G$. Let $f$, $F$, $\mathcal{F}$ satisfy \eqref{3.bound_inhom}.
Let us assume that
\begin{equation}
   \label{2-page13-par3}
\|u\|_{L^2(B_{r_0}(x_0))}\leq\eta,
\end{equation}
\begin{equation}
   \label{3-page13-par3}
\|u\|_{L^2(\Omega)}\leq E_0,
\end{equation}
for given $\eta>0$, $E_0>0$. We have
\begin{equation}
   \label{4-page13-par3}
\|u\|_{L^2(G)}\leq C(\epsilon+\eta)^\delta(E_0+\epsilon+\eta)^{1-\delta},
\end{equation}
where
\begin{equation}
   \label{5-page13-par3}
   C=C_1\left(\frac{|\Omega|}{r^2}\right)^\frac{1}{2},
\end{equation}
\begin{equation}
   \label{6-page13-par3}
   \delta\geq\alpha^{\frac{C_2|\Omega|}{r^2}},
\end{equation}
with $C_1>0$ and $\alpha$, $0<\alpha<1$, only depending on
$\gamma$, $M$ and $\delta_1$, and with $C_2$ only
depending on $\gamma$ and $\delta_1$, where
$\delta_1=\min_{\overline{G}}{\mathcal{D}}$.
\end{theo}
\begin{proof}
The proof is essentially based on an iterated application of the three sphere inequality, see
\cite[Proof of Theorem 5.1]{A-R-R-V} for details.
\end{proof}

\begin{theo} [Local stability for the Cauchy problem]
\label{theo:LSC}
Let $u\in H^2(\Omega)$ be  a weak solution to the Cauchy problem
\eqref{eq:1-page7-par3}--\eqref{eq:5-page7-par3},
where $\mathbb{P}$, defined by \eqref{3.P}, satisfies \eqref{eq:sym-conditions-C-components},
   \eqref{eq:3.bound_quantit}, \eqref{eq:3.convex} and the dichotomy condition
in the rectangle $R_{\frac{\rho_0}{M_0},\rho_0}(P_0)$, $\Sigma$
satisfies \eqref{1-page6-par3}, $f$, $F$, $\mathcal{F}$ satisfy
\eqref{3.bound_inhom}, and $g_1$, $g_2$, $\widehat{M}$ satisfy
\eqref{2-page6-par3}. Assuming the a priori bound
\begin{equation}
\label{1-page14-par3}
\|u\|_{L^2(\Omega)}\leq E_0,
\end{equation}
then
\begin{equation}
\label{2-page14-par3}
\|u\|_{L^2\left(R_{\frac{\rho_0}{2M_0},\frac{\rho_0}{2}}(P_0)\cap\Omega\right)}\leq
C(\epsilon+\eta)^\delta(E_0+\epsilon+\eta)^{1-\delta},
\end{equation}
where $C>0$ and $\delta$, $0<\delta<1$, only depend on
$\gamma$, $M$, $M_0$, $M_1$ and on
$\delta_1=\min_{\overline{\mathcal{O}}}{\mathcal{D}}$, where
$\mathcal{O}= R_{\frac{\rho_0}{M_0},\rho_0}(P_0)$.
\end{theo}

\begin{proof}
Representing locally $\Omega$ in a neighborhood of $P_0$ as
\begin{equation*}
\Omega\cap R_{\frac{\rho_0}{M_0},\rho_0}(P_0)=\{(x_1,x_2)\in R_{\frac{\rho_0}{M_0},\rho_0}\ |\ x_2>\psi(x_1)\},
\end{equation*}
let
\begin{equation*}
r_0=\frac{\rho_0}{2(\sqrt{1+M_0^2}+1)},
\end{equation*}
\begin{equation*}
x_0=\left(0,r_0-\frac{\rho_0}{2}\right).
\end{equation*}
We have that
\begin{equation*}
B_{r_0}(x_0)\subset R^-_{\frac{\rho_0}{2M_0},\frac{\rho_0}{2}}(P_0),
\end{equation*}
so that
\begin{equation*}
\|u\|_{L^2(B_{r_0}(x_0))}\leq C\eta.
\end{equation*}
The thesis easily follows by applying Theorem \ref{theo:HPS}
with $\Omega=R_{\frac{\rho_0}{M_0},\rho_0}(P_0)$, $G=R_{\frac{\rho_0}{2M_0},\frac{\rho_0}{2}}(P_0)$,
$h=\frac{r_0}{2}$.
\end{proof}

\section{Carleman estimate for second order elliptic operators\label{Sec4.1}}

In this and in the next section we consider $n\geq2$, where $n$ is the
space dimension. Moreover, in this section we use a notation for euclidean norm and scalar product which differs
{}from the standard one used in the other sections.

Let
\begin{equation}
\label{1-1}
Pu=\partial_{i}(g^{ij}(x)\partial_{i}u)
\end{equation}
where $\{{ g^{ij}(x)}\} _{i,j=1}^{n}$ is a symmetric matrix valued
function which satisfies a uniform ellipticity condition and whose
entries are Lipschitz continuous functions. In order to simplify
the calculations, in the sequel we shall use some standard
notations in Riemannian geometry, but always dropping the
corresponding volume element in the definition of the
Laplace-Beltrami metric. More precisely, denoting by
$g(x)=\{g_{ij}(x)\}_{i,j=1}^{n}$ the inverse of the matrix
$\{g^{ij}(x)\}_{i,j=1}^{n}$ we have $g^{-1}(x)=\{g^{ij}(x)
\}_{i,j=1}^{n}$ and we use the following notation when considering
either a smooth function $v$ or two vector fields $\xi $ and
$\eta$

i. $\xi \cdot \eta =\sum\limits_{i,j=1}^{n}g_{ij}(x)\xi _{i}\eta
_{j}$, \quad $| \xi|^{2}=\sum\limits_{i,j=1}^{n}g_{ij}(x)\xi _{i}\xi
_{j},$

ii. $\nabla v=( \partial_{1}v,...\partial _{n}v) $, \quad $\nabla
_{g}v(x)=g^{-1}(x)\nabla v(x)$, \\ $\divrg(\xi)=\sum\limits_{i=1}^{n}\partial_{i}\xi _{i},
\quad \Delta_{g}v=\divrg(\nabla_{g}v)$,

iii. $(\xi ,\eta )_{n}=\sum\limits_{i=1}^{n}\xi _{i}\eta _{i}$, $
| \xi| _{n}^{2}=\sum\limits_{i=1}^{n}\xi _{i}^{2}$.

\bigskip

With this notation the following formulae hold true when $u$, $v$
and $w$ are smooth functions
\begin{equation}
\label{1-2} Pu=\Delta _{g}u\text{, }\quad \Delta _{g}\left( v^{2}\right)
=2v\Delta _{g}v+2\left\vert \nabla _{g}v\right\vert ^{2}
\end{equation}%
and
\begin{equation}
\label{2-2}\int_{\mathbb{R}^{n}}v\Delta
_{g}wdx=\int_{\mathbb{R}^{n}}w\Delta
_{g}vdx=-\int_{\mathbb{R}^{n}}\nabla _{g}v\cdot \nabla _{g}wdx.
\end{equation}%
We shall also use the following Rellich identity
\begin{eqnarray}
\label{3-2}2(B\cdot \nabla _{g}v)\Delta _{g}v =\divrg\left( 2(B\cdot
\nabla _{g}v)\nabla _{g}v-B|\nabla _{g}v| ^{2}\right)+\\+(\divrg B)
|\nabla _{g}v| ^{2}-2\partial _{i}B^{k}g^{ij}\partial
_{j}v\partial _{k}v+B^{k}\partial _{k}g^{ij}\partial _{i}v\partial
_{j}v\text{ ,}\nonumber
\end{eqnarray}%
where $B=(B^{1},...,B^{n})$ is a smooth vector field.

We denote by
$w\in C^2(\mathbb{R}^{n}\setminus\{0\})$ a function that we shall
choose later on such that $w(x)>0$ and $|\nabla _{g}w|>0$ in
$\mathbb{R}^{n}\setminus\{0\}$.

Given $f\in C^\infty(\mathbb{R}^{n}\setminus \{0\})$, let us set
\begin{equation}
\label{10-3}P_{\tau}(f)=w^{-\tau}P(w^{\tau}f ),
\end{equation}
\begin{equation}
    \label{50-3}
A_w(f)=\frac{w}{|\nabla_{g}w|}
\partial_{Y} f+\frac{1}{2}F_{w}^gf,
\end{equation}
where
\begin{equation}
\label{20-3} F_{w}^{g}=\frac{w\Delta
_{g}w-|\nabla_{g}w|^2}{|\nabla_{g}w|^2},
\end{equation}

\begin{equation}
\label{30-3}Y=\frac{\nabla_{g}w}{|\nabla_{g}w|},
\end{equation}

\begin{equation}
\label{40-3}\partial_Y f=\nabla_g f\cdot Y.
\end{equation}
\bigskip

With the notation introduced above we have%
\begin{equation}
\label{1-3}P_{\tau }(f)=P_{\tau }^{(s)}(f)+P_{\tau }^{(a)}(f),
\end{equation}%
where $P_{\tau }^{(s)}$ and $P_{\tau }^{(a)}$ are the symmetric
and the antisymmetric part of the operator $P_{\tau }$ with
respect to the $L^2$ scalar product, respectively. \\More precisely
we have
\begin{equation}
\label{2-3}P_{\tau }^{(s)}(f)=\Delta _{g}f+\tau
^{2}\frac{\left\vert \nabla _{g}w \right\vert ^{2}}{w^2}f
\end{equation}%
and%

\begin{equation}
\label{3-3}P_{\tau }^{(a)}(f)=2\tau \frac{\left\vert \nabla _{g}w
\right\vert ^{2}}{w^2} A_w(f).
\end{equation}%

Moreover, let us denote by $S^g_{w}$ the symmetric matrix
$S^g_{w}=\{S_{w}^{g,ij}\}_{i,j=1}^{n}$, where
\begin{equation}
\label{1-4} S_{w}^{g,ij}=\frac{1}{2}\left((\divrg B)-F_{w}^g)
g^{ij}\\-\partial _{k}B^{j}g^{ki}-\partial
_{k}B^{i}g^{kj}+B^{k}\partial _{k}g^{ij}\right),
\end{equation}
with
\begin{equation}
\label{2-4} B= \frac{w}{\left\vert \nabla _{g}w \right\vert
}Y=\frac{w\nabla _{g}w }{\left\vert \nabla _{g}w \right\vert
^{2}}.
\end{equation}
We also denote
\begin{equation}
\label{3-4} \mathcal{M}_{w}^{g}=S_{w}^g g.
\end{equation}
Notice that
\begin{equation}
\label{4-4} \mathcal{M}_{w}^{g}\xi \cdot \eta =\xi \cdot
\mathcal{M}_{w}^{g}\eta, \quad \text{for every }\xi,\eta \in
\mathbb{R}^{n}
\end{equation}%
and, letting $\xi _{g}=g^{-1}\xi $, $\eta _{g}=g^{-1}\eta $,
\begin{equation}
\label{5-4} \mathcal{M}_{w}^{g}\xi _{g}\cdot \eta _{g}=(S_{w}^g\xi
,\eta )_{n},\quad \text{for every }\xi,\eta \in \mathbb{R}^{n}.
\end{equation}
\bigskip

The proof of the following lemma is straightforward.
\bigskip
\begin{lem}
\label{lem:1-4.1}Let $v\in C^2(\mathbb{R}^n\setminus\{0\})$ be a
function that satisfies the conditions $v(x)>0$,
$|\nabla_gv(x)|>0$ for every $x\in\mathbb{R}^n\setminus\{0\}$. Let
$S_{v}^{g}$, $\mathcal{M}_{v}^{g}$, $F_{v}^{g}$ and $B$ be
obtained substituting $w$ with $v$ in the \eqref{1-4},
\eqref{3-4}, \eqref{20-3} and \eqref{2-4}, respectively.

Let $\varphi\in C^2(0,+\infty)$ be such that $\varphi(s)>0$,
$\varphi'(s)>0$, for every $s\in(0,+\infty) $. Let us denote
\begin{equation}
\label{1-5}\Phi(s)=\frac{\varphi(s)}{s\varphi'(s)}.
\end{equation}
We have
\begin{equation}
\label{2-5} \mathcal{M}_{v}^{g}\nabla_g v=S_{v}^g\nabla v=0,
\end{equation}
\begin{equation}
\label{3-5} F_{\varphi(v)}^{g}=\Phi(v)F_{v}^{g}-\Phi'(v)v ,
\end{equation}

\begin{equation}
\label{4-5} \mathcal{M}_{\varphi(v)}^{g}\xi \cdot \eta =
v\Phi'(v)\left(\xi\cdot\eta-
\frac{(\nabla_gv\cdot\xi)(\nabla_gv\cdot\eta)}{|\nabla_{g}v|^2}\right)+\Phi(v)\mathcal{M}_{v}^{g}
\xi \cdot\eta .
\end{equation}

\end{lem}

\bigskip
In the sequel we shall use the following notation

\begin{equation}
\label{1*-6}\nabla_g^N f=(\nabla_gv\cdot\nabla_g f) \frac{\nabla_g
v}{|\nabla_g v|^2}=(\partial_Y f\cdot Y)Y,
\end{equation}

\begin{equation}
\label{2*-6}\nabla_g^T f=  \nabla_g f-\nabla_g^N f,
\end{equation}
Notice that $\nabla_g^N f$ and $\nabla_g^T f$ are the normal
component and the tangential component (with respect to the
Riemannian metric $\{g_{ij}\}_{i,j=1}^n$) of $\nabla_g f$ to the
level surface of $w$ respectively. In particular $\nabla_g^N f$
and $\nabla_g^T f$ are invariant with respect to
transformations of the type $\widetilde{w}=\varphi(w)$, where
$\varphi$ satisfies the hypotheses of Lemma \ref{lem:1-4.1}. We
have
\bigskip
\begin{equation}
\label{3*-6}\nabla_g^T f\cdot Y=0,\quad \nabla_g f=\nabla_g^N
f+\nabla_g^T f,
\end{equation}
\begin{equation}
\label{4*-6} |\nabla_g f|^2=|\nabla_g^N f|^2 +|\nabla_g^T
f|^2=(\partial_Y f)^2+|\nabla_g^T f|^2,
\end{equation}
\begin{equation}
\label{5*-6} \nabla_g^N f\cdot\nabla_g^T f=0.
\end{equation}

\noindent
In addition, observe that by \eqref{4-4} and \eqref{2-5} we have
\begin{equation}
\label{1-6}\mathcal{M}_w^g\nabla_g f\cdot\nabla_g
f=\mathcal{M}_w^g\nabla_g^T f\cdot\nabla_g^T f.
\end{equation}%

\begin{lem}
\label{lem:2-4.1} Let $w\in C^2(\mathbb{R}^n\setminus\{0\})$ be
such that $w(x)>0$, $|\nabla_gw(x)|>0$ for every
$x\in\mathbb{R}^n\setminus\{0\}$. For every $\tau\neq 0$ we have
\begin{multline}
\label{2-6}
\frac{w^2}{|\nabla_gw|^2}\left(P_\tau(f)\right)^2=\frac{w^2}{|\nabla_gw|^2}\left(P^{(s)}_\tau(f)\right)^2
+4{\tau}^2\left(\partial_Y
f\right)^2\left(1+(2\tau)^{-1}F_{w}^{g}\right)+\\
+4\tau\left(\mathcal{M}_{w }^{g}\nabla_g^T f\cdot\nabla_g^T
f+\frac{1}{2}F_{w }^{g}|\nabla_g^T f|^2\right)-\\
-2\tau^3\frac{|\nabla_gw|^2}{w^2}F_{w
}^{g}\left(1+(2\tau)^{-1}F_{w}^{g}\right)f^2+ 2\tau F_{w }^{g}f
P_\tau(f)+ \divrg(q),
\end{multline}
where
\begin{eqnarray}
   \label{2-7}
q=\frac{2\tau w}{|\nabla_gw|}\left(2(\partial_Y f)\nabla_g
f-|\nabla_g f|^2Y+\tau^2f^2\frac{|\nabla_gw|^2}{w^2}Y\right).
\end{eqnarray}
\end{lem}
\begin{proof}
By \eqref{1-3} we have
\begin{multline}
\label{3-7} \frac{w^2}{|\nabla_gw|^2}\left(P_{\tau
}(f)\right)^2=\frac{w^2}{|\nabla_gw|^2}\left(P_{\tau
}^{(s)}(f)\right)^2+\\+2\frac{w^2}{|\nabla_gw|^2}P^{(s)}_\tau(f)P^{(a)}_\tau(f)+\frac{w^2}{|\nabla_gw|^2}\left(P_{\tau
}^{(a)}(f)\right)^2.
\end{multline}
Let us consider the second term at the right-hand side of
\eqref{3-7}. We have

\begin{multline}
\label{10-7}
2\frac{w^2}{|\nabla_gw|^2}P^{(s)}_\tau(f)P^{(a)}_\tau(f)=4\tau\left(\Delta_g
f+\tau^2\frac{|\nabla_gw|^2}{w^2}f\right)A_w(f)=\\
=4\tau\left(\frac{w\nabla_gw\cdot\nabla_gf}{|\nabla_gw|^2}\right)\Delta_g
f+2\tau F_{w}^{g}f\Delta_g
f+4\tau^3\frac{|\nabla_gw|^2}{w^2}A_w(f)f=\\=4\tau\left(\frac{w\nabla_gw\cdot\nabla_gf}{|\nabla_gw|^2}\right)\Delta_g
f+2\tau F_{w}^{g}f\Delta_g
f+2\tau^3\divrg\left(\frac{\nabla_gw}{w}f^2\right).
\end{multline}
Now we transform the term
$4\tau\left(\frac{w\nabla_gw\cdot\nabla_gf}{|\nabla_gw|^2}\right)\Delta_g
f$ by applying the Rellich identity \eqref{3-2} with
$B=\frac{w\nabla_gw}{|\nabla_gw|^2}$ and $v=f$. We obtain

\begin{multline}
\label{1-8}
2\frac{w^2}{|\nabla_gw|^2}P^{(s)}_\tau(f)P^{(a)}_\tau(f)=\\
=4\tau\mathcal{M}^g_w\nabla_g f\cdot\nabla_g f +2\tau
F_{w}^{g}|\nabla_g f|^{2}+2\tau F_{w}^{g} f\Delta_g f+\divrg(q),
\end{multline}
where $q$ is given by \eqref{2-7}.

Now we transform the third term at the right-hand side of
\eqref{1-8} by using the following trivial consequence of
\eqref{1-3}
\begin{equation}
\label{3-8}
\Delta_gf=P_\tau(f)-\tau^2\frac{|\nabla_gw|^{2}}{w^2}f-2\tau\frac{|\nabla_gw|^{2}}{w^2}A_w(f)
\end{equation}
and we obtain
\begin{multline}
\label{1-9} 2\tau F_{w}^{g} f\Delta_g f= 2\tau
F_{w}^{g}fP_\tau(f)-\\- 2\tau^3\frac{|\nabla_gw|^2}{w^2}F_{w
}^{g}\left(1+\frac{1}{\tau}F_{w}^{g}\right)f^2-
4\tau^2\frac{|\nabla_gw|}{w}F_{w}^{g}f\partial_Yf.
\end{multline}
Now, just spreading the square in the third term at
the right-hand side of \eqref{3-7}, we have
\begin{multline}
  \label{2-9}
\frac{w^2}{|\nabla_gw|^2}\left(P_{\tau }^{(a)}(f)\right)^2=\\
=4\tau^2(\partial_Yf)^2+\tau^2\frac{|\nabla_gw|^2}{w^2}\left(F_{w
}^{g}\right)^2f^2 +4\tau^2\frac{|\nabla_gw|}{w}F_{w
}^{g}f\partial_Yf,
\end{multline}
so that, by \eqref{4*-6}, \eqref{1-6}, \eqref{3-7}, \eqref{1-8},
\eqref{1-9} and \eqref{2-9} we obtain identity  \eqref{2-6}.
\end{proof}
\bigskip
In the sequel of this section we assume that the matrix $\{{
g^{ij}(x)}\} _{i,j=1}^{n}$ satisfies the following conditions
\begin{equation}
\label{1-10}
\lambda|\xi|_n^2\leq\sum_{i,j=1}^ng^{ij}(x)\xi_i\xi_j\leq\lambda^{-1}|\xi|_n^2,\quad
\text { for every }x\in\mathbb{R}^n, \xi\in\mathbb{R}^n
\end{equation}
and
\begin{equation}
\label{2-10}
\sum_{i,j=1}^n|g^{ij}(x)-g^{ij}(y)|\leq\Lambda|x-y|_n, \quad\text { for
every }x\in\mathbb{R}^n, y\in\mathbb{R}^n,
\end{equation}
where $\lambda\in(0,1]$ and $\Lambda>0$.
\bigskip
Now we introduce some additional notation that we shall use in the
sequel. Let $\Gamma=\{{ \gamma_{ij}}\} _{i,j=1}^{n}$ be a matrix
that we shall choose later on. We assume that
\begin{equation}
\label{4-10} m_\ast|x|_n^2\leq\left(\Gamma x,x\right)_n\leq
m^{\ast}|x|_n^2, \quad\text { for every }x\in\mathbb{R}^n,
\end{equation}
where $m_\ast$ and $m^{\ast}$ are the minimum and the maximum
eigenvalue of $\Gamma$ respectively, and $m_\ast>0$. Let us denote
\begin{equation}
\label{10-10} \sigma(x)=\left(\left(\Gamma
x,x\right)_n\right)^{1/2}
\end{equation}
and we denote
\begin{equation}
\label{*-11} S^{(0)}=S_{\sigma}^{g(0)},
\end{equation}
where we recall that
\begin{equation}
\label{1-11}
S_\sigma^{g(0),ij}=\frac{1}{2}\left((\divrg B_0)-F_{\sigma}^{g(0)})
g^{ij}(0)\\-\partial _{k}B_0^{j}g^{ki}(0)-\partial
_{k}B_0^{i}g^{kj}(0)\right)
\end{equation}
and
\begin{equation}
\label{10-11}
B_0=\{B_0^i\}_{i=1}^n=\left\{\frac{\sigma(x)g^{ij}(0)\partial_j\sigma(x)}{g^{lm}(0)\partial_l\sigma(x)
\partial_m\sigma(x)}\right\}_{i=1}^n,
\end{equation}
\bigskip
\begin{equation}
\label{20-11}
F_{\sigma}^{g(0)}=\frac{\sigma(x)g^{ij}(0)\partial^2_{ij}\sigma(x)-g^{ij}(0)\partial_i\sigma(x)\partial_j\sigma(x)}
{g^{ij}(0)\partial_i\sigma(x)\partial_j\sigma(x)}.
\end{equation}
\bigskip
Moreover, for any fixed $\xi\in\mathbb{R}^n$, $(S^{(0)}\xi,\xi)_n$
is an homogeneous function with respect to the $x$ variable of
degree $0$, hence the following number is well defined
\begin{equation}
\label{2-11}
\omega_0=\sup\left\{-(S^{(0)}\xi,\xi)_n\ |\ g^{ij}(0)\xi_i\xi_j=1,\
g^{ij}(0)\partial_i\sigma(x)\xi_j=0,\
x\in\mathbb{R}^n\setminus{0}\right\}.
\end{equation}

\bigskip

We observe that $\omega_0$ is a nonnegative number. More precisely
we have the following proposition.

\begin{prop}
\label{propRem} Let
$Q=\sqrt{g(0)}\Gamma^{-1}\sqrt{g(0)}$,
where $\sqrt{g(0)}$ is the positive square root of the
matrix $g(0)$. Let $\varrho_{\ast}$ and $\varrho^{\ast}$ be the
minimum and the maximum eigenvalues of the matrix $Q$
respectively. Then the following equality holds true
\begin{equation}
\label{1-11bis} \omega_0=\frac{\varrho^{\ast}}{\varrho_{\ast}}-1.
\end{equation}
\end{prop}
\begin{proof}
In order to prove \eqref{1-11bis}, let us denote
\begin{equation}
\label{2-11bis} K=\Gamma g^{-1}(0)\Gamma
\end{equation}
and let us notice that, with the conditions
\begin{equation}
\label{3-11bis} (g^{-1}(0)\xi,\xi)_n=1 \quad
(g^{-1}(0)\nabla\sigma(x),\xi)_n=0
\end{equation}
and with the normalization condition

\begin{equation}
\label{1-11ter} (Kx,x)_n=1,
\end{equation}
we have
\begin{equation}
\label{2-11ter} -(S^{(0)} \xi, \xi)_n=(\Gamma x,
x)_n\left((K\Gamma^{-1}K x, x)_n +
    (g^{-1}(0)\Gamma g^{-1}(0) \xi,\xi)_n\right)-2.
\end{equation}
Moreover, by introducing the new variables

\begin{equation}
\label{3-11ter} \eta=\left(\sqrt{g(0)}\right)^{-1}\xi,
\quad y=\left(\sqrt{g(0)}\right)^{-1}\Gamma x,
\end{equation}
conditions \eqref{3-11bis} and \eqref{1-11ter} become respectively

\begin{equation}
\label{4-11ter} |\eta|_n^2=1, \quad (y,\eta)_n=0,
\end{equation}
and
\begin{equation}
\label{5-11ter} |y|_n^2=1
\end{equation}
so that expression \eqref{2-11ter} is equal to
\begin{equation}
\label{6-11ter} H(y,\eta):=(Qy,y)_n\left((Q^{-1}y,y)_n +
    (Q^{-1}\eta,\eta)_n\right)-2.
\end{equation}
Thus we have
\begin{equation}
\label{10-11ter} \omega_0=\sup\left\{H(y,\eta)\ |\
|y|_n=1,|\eta|_n=1, (y,\eta)_n=0\right\}.
\end{equation}

\bigskip

Now let $z_{\ast}$ and $z^{\ast}$ be two linearly independent unit
eigenvectors of $Q$ such that $Qz_{\ast}=\varrho_{\ast}z_{\ast}$
and $Qz^{\ast}=\varrho^{\ast}z^{\ast}$. We have
\begin{equation}
\label{1-11quater}
H(z^{\ast},z_{\ast})=\frac{\varrho^{\ast}}{\varrho_{\ast}}-1,
\end{equation}
hence

\begin{equation}
\label{2-11quater}
\omega_0\geq\frac{\varrho^{\ast}}{\varrho_{\ast}}-1.
\end{equation}

\bigskip

In order to complete the proof of \eqref{1-11bis} we need to prove
that

\begin{equation}
\label{3-11quater}
\omega_0\leq\frac{\varrho^{\ast}}{\varrho_{\ast}}-1.
\end{equation}
To this aim we recall the following Kantorovich inequality
 \cite{Ka}, \cite{Mi}. Let $\mathcal{A}$ be a $m\times m$
positive definite symmetric real matrix and let $\alpha_{\ast}$,
$\alpha^{\ast}$ be the minimum and the maximum eigenvalues of
$\mathcal{A}$ respectively, then for every $X\in\mathbb{R}^m$ we
have
\begin{equation}
    \label{4-11quater}
    (\mathcal{A}X,X)_m(\mathcal{A}^{-1}X,X)_m\leq
    \frac{1}{4}\left(\sqrt{\frac{\alpha^*}{\alpha_*}}+\sqrt{\frac{\alpha_*}{\alpha^*}}\right)^2|X|_m^4.
\end{equation}

\bigskip

Now let $m=2n$, $X=(y,\eta)^t$ and
\begin{equation}
    \label{10-11pentium}
\mathcal{A}=\left(
\begin{array}{cc}
Q & 0 \\
0 & Q%
\end{array}%
\right),
\end{equation}
we have, for every $y,\eta\in \mathbb{R}^n$ such that
$|y|_n=|\eta|_n=1, (y,\eta)=0$
\begin{equation}
\label{1-11pentium} H(y,\eta)
=(\mathcal{A}X,X)_{2n}(\mathcal{A}^{-1}X,X)_{2n}-(Q\eta,\eta)_n(\mathcal{A}^{-1}X,X)_{2n}-2.
\end{equation}

By Schwarz inequality we have

\begin{multline}
\label{2-11pentium}\qquad (Q\eta,\eta)_n(\mathcal{A}^{-1}X,X)_{2n}
= (Q\eta,\eta)_n(Q^{-1} y,y)_n+\\+(Q\eta,\eta)_n(Q^{-1}
\eta,\eta)_n\geq\frac{\varrho_{\ast}}{\varrho^{\ast}}+|\eta|_n^2=\frac{\varrho_{\ast}}{\varrho^{\ast}}+1.
\end{multline}
On the other hand, the first term on the right-hand side of
\eqref{1-11pentium} can be estimated {}from above by inequality
\eqref{4-11quater}. By the obtained inequality and by
\eqref{2-11pentium} we get \eqref{3-11quater}, that completes the
proof of \eqref{1-11bis}.
\end{proof}

\bigskip

In the next Lemma and in the sequel we shall use the following
notation when dealing with a matrix $A=\{a_{ij}\}_{i,j=1}^n$
\begin{equation}
\label{10N-11} \left\vert
A\right\vert=\left(\sum_{i,j=1}^na_{ij}^2\right)^{1/2}.
\end{equation}

\begin{lem}
\label{lem:3-4.1} There exists a constant $C,C\geq1,$ depending
only on $\lambda,\Lambda,m_\ast$ and $m^{\ast}$ such that for
every $x\in\mathbb{R}^n\setminus\{0\}$, $0<\sigma(x)\leq1$, the
following inequalities hold true
\begin{equation}
\label{1-12} C^{-1}\leq\left\vert \nabla _{g}\sigma
\right\vert\leq C,\text { } \left\vert F _{\sigma}^g
\right\vert\leq C, \text { } \left\vert S ^{(0)} \right\vert\leq
C,
\end{equation}
\begin{equation}
\label{2-12}  \left\vert F _{\sigma}^g -F
_{\sigma}^{g(0)}\right\vert\leq C\sigma, \text { } \left\vert
S_\sigma ^{g}- S ^{(0)} \right\vert\leq C\sigma,
\end{equation}
\begin{equation}
\label{3-12} \mathcal{M}_w^g\nabla_g^T f\cdot\nabla_g^T
f\geq-(\omega_0+C\sigma)\left\vert \nabla _{g}^Tf\right\vert^2.
\end{equation}
\end{lem}
\begin{proof}
The proof of \eqref{1-12} and \eqref{2-12} is straightforward. We
prove inequality \eqref{3-12}. Denote by
\begin{equation}
\label{4-12}\zeta=g\nabla_g^T f.
\end{equation}
We have by \eqref{1-10}, \eqref{2-10}, \eqref{2-12} and
\eqref{4-12}
\begin{multline}
\label{5-12} \mathcal{M}_w^g\nabla_g^T f\cdot\nabla_g^T
f=(S_{\sigma}^g\zeta,\zeta)_n\geq\\\geq(S^{(0)}\zeta,\zeta)_n-\left\vert
((S_{\sigma}^g-S^{(0))})\zeta,\zeta)_n\right\vert\geq(S^{(0)}\zeta,\zeta)_n-C\sigma\left\vert
\nabla _{g}^Tf\right\vert^2,
\end{multline}
where $C$ depends only on $\lambda,\Lambda,m_\ast$ and $m^{\ast}$.

\bigskip

Now, let us consider the term $(S_{\sigma}^g\zeta,\zeta)_n$
 on the right-hand side of
\eqref{5-12}. Denoting by
\begin{equation}
\label{10-13}\tilde{\zeta
}=\zeta+g(0)\left(g^{-1}(x)-g^{-1}(0)\right)\zeta,
\end{equation}
we have $g^{-1}(0)\tilde{\zeta }=g^{-1}(x)\zeta=\nabla_g^Tf$,
hence
\begin{equation}
\label{2-13}g^{ij}(0)\tilde{\zeta
}_j\partial_i\sigma=\nabla_g^Tf\cdot\nabla_g\sigma=0.
\end{equation}
In addition we have
\begin{equation}
\label{3-13}|\zeta-\tilde\zeta|_n\leq C|\nabla_g^Tf|\sigma
\end{equation}
and
\begin{equation}
\label{4-13}g^{ij}(0)\tilde{\zeta
}_j\tilde\zeta_i\leq\left(1+C\sigma\right)|\nabla_g^Tf|^2,
\end{equation}
where $C$ depends only on $\lambda,\Lambda,m_\ast$ and $m^{\ast}$.
\bigskip
By \eqref{2-11}, \eqref{1-12}, \eqref{2-13} and \eqref{3-13}, we
obtain, for every $x\in\mathbb{R}^n\setminus\{0\}$ such that
$0<\sigma(x)\leq1$,
\begin{multline}
\label{20-13}
(S^{(0)}\zeta,\zeta)_n\geq(S^{(0)}\tilde\zeta,\tilde\zeta)_n-\\-\left\vert
(S^{(0)}(\zeta-\tilde\zeta),\zeta-\tilde\zeta)_n \right\vert\ -
2\left\vert(S^{(0)}(\zeta-\tilde\zeta),\tilde\zeta)_n
\right\vert\geq\\\geq-\omega_0(g^{-1}(0)\tilde\zeta,\tilde\zeta)_n-C|\zeta-\tilde\zeta|_n^2-
2C|\zeta-\tilde\zeta|_n|\tilde\zeta|_n\geq\\\geq-(\omega_0+C\sigma)|\nabla_g^Tf|^2,
\end{multline}
where $C$ depends only on $\lambda,\Lambda,m_\ast$ and $m^{\ast}$.
By the just obtained inequality and by \eqref{5-12} we obtain
\eqref{3-12}.
\end{proof}

\bigskip

Let $r$ be a given positive number, in the sequel we shall denote
by $B_r^{\sigma}$ the set
$\left\{x\in\mathbb{R}^n|\sigma(x)<r\right\}$. In addition, in
order to simplify the notation, we shall denote
$\int_{\mathbb{R}^n}(.)dx$ simply by $\int$ and, instead to write
\textquotedblleft$f$ is a function that belongs to
$C_0^\infty\left(\mathbb{R}^n\setminus\{0\}\right)$ and $f$ is
such that $\text { supp}(f)\subset
B_r^{\sigma}\setminus\{0\}$\textquotedblright, we shall write
simply \textquotedblleft$f\in C_0^\infty\left(
B_r^{\sigma}\setminus\{0\}\right)$\textquotedblright.

\bigskip

\begin{theo}
\label{theo:4-4.1}Let $\beta$ be a number such that
$\beta>\omega_0$, let
\begin{equation}
\label{1-15} \varphi(s)=e^{-s^{-\beta}}
\end{equation}
and let $w(x)=\varphi\left(\sigma(x)\right)$. There exist
constants $C$, $\tau_1$ and $r_0$, ($C\geq 1$, $\tau_1\geq 1$,
$0<r_0\leq 1$) depending only on $\lambda,\Lambda,m_\ast,m^{\ast}$
and $\beta$ such that for every $u\in C_0^\infty\left(
B_{r_0}^{\sigma}\setminus\{0\}\right)$ and for every
$\tau\geq\tau_1$ the following inequality holds true
\begin{multline}
\label{2-15}
\tau\int\sigma^{\beta}w^{-2\tau}|\nabla_gu|^2+\tau^3\int\sigma^{-\beta-2}w^{-2\tau}u^2\leq
C\int\sigma^{2\beta+2}w^{-2\tau}\left(\Delta_gu\right)^2.
\end{multline}

\end{theo}

\begin{proof}Let $w(x)=\varphi\left(\sigma(x)\right)$, where $\sigma(x)=\left((\Gamma
x,x)_n\right)^{1/2}$. Let us notice that $\varphi$ satisfies the
hypotheses of Lemma \ref{lem:1-4.1} and that
\begin{equation}
\label{1-16}\Phi(s)=\frac{s^{\beta}}{\beta}.
\end{equation}
Let $u\in C_0^\infty\left( B_{1}^{\sigma}\setminus\{0\}\right)$
and $f=w^{-\tau}u$. By \eqref{4-5} and by \eqref{3-12} we have
\begin{equation}
\label{2-16} \mathcal{M}_w^g\nabla_g^T f\cdot\nabla_g^T
f\geq{\sigma}^{\beta}\left(1-\frac{\omega_0}{\beta}-C\sigma\right)\left\vert
\nabla _{g}^Tf\right\vert^2,
\end{equation}
where $C$ depends only on $\lambda,\Lambda,m_\ast,m^{\ast}$ and
$\beta$.\\
Now, denoting
\begin{equation}
\label{10-16}
\psi_0={\sigma}^{\beta}\left(-1+\frac{1}{\beta}F_{\sigma}^{g(0)}\right),
\end{equation}
by \eqref{3-5} we have

\begin{equation}
\label{20-16}
F_{w}^{g}=\psi_0+\frac{{\sigma}^{\beta}}{\beta}\left(F_{\sigma}^{g}-F_{\sigma}^{g(0)}\right),
\end{equation}
hence by \eqref{1-12} and \eqref{2-12} of Lemma \ref{lem:3-4.1} we
have, for every $x\in B_1^{\sigma}\setminus\{0\}$,
\begin{equation}
\label{2*-16} \left\vert F_{w}^{g}\right\vert \leq
C{\sigma}^{\beta}, \quad\quad\quad \left\vert
F_{w}^{g}-\psi_0\right\vert\leq C{\sigma}^{\beta+1},
\end{equation}
where $C, C\geq1$, depends only on
$\lambda,\Lambda,m_\ast,m^{\ast}$ and $\beta$.

\bigskip

Let $\psi_1$ be a function that we shall choose later on, by
\eqref{2-3} we have
\begin{multline}
\label{1-17} \qquad \frac{w^2}{|\nabla_gw|^2}\left(P_{\tau
}^{(s)}(f)\right)^2=\\=\frac{w^2}{|\nabla_gw|^2}\left(P_{\tau
}^{(s)}(f)-\tau\frac{|\nabla_gw|^2}{w^2}\psi_1f+\tau\frac{|\nabla_gw|^2}{w^2}\psi_1f\right)^2 \geq\\
\geq 2\tau\psi_1f\left(P_{\tau
}^{(s)}(f)-\tau\frac{|\nabla_gw|^2}{w^2}\psi_1f\right)=\\=
2\tau^3\left(\left(1-\frac{\psi_1}{\tau}
\right)\psi_1\frac{|\nabla_gw|^2}{w^2}+\frac{1}{2\tau^2}\Delta_g\psi_1\right)f^2
-2\tau\psi_1|\nabla_gf|^2+\divrg(q_1),
\end{multline}
where
\begin{equation}
\label{10-17}q_1=\tau\left(2\psi_1f\nabla_gf-f^2\nabla_g\psi_1\right).
\end{equation}
By inequalities \eqref{2-16} and \eqref{1-17}, by \eqref{4*-6} and by Lemma
\ref{lem:2-4.1} we obtain
\begin{multline}
\label{2-17} \frac{w^2}{|\nabla_gw|^2}\left(P_\tau(f)\right)^2\geq
2\tau^3a_1f^2+\\+4\tau
a_2|\nabla_g^Tf|^2+4{\tau}^2a_3\left(\partial_Y f\right)^2 + 2\tau
F_{w }^{g}f P_{\tau}(f)+ \divrg(q_2),
\end{multline}
where

\begin{equation}
\label{30-18}
a_1=\frac{|\nabla_gw|^2}{w^2}\left(\left(\psi_1-F_w^g\right)-\frac{1}{\tau}\left(\frac{1}{2}\left(F_w^g\right)^2+
\psi_1^2\right)\right)+\frac{1}{2\tau^2}\Delta_g\psi_1,
\end{equation}

\begin{equation}
\label{20-18}
a_2={\sigma}^{\beta}\left(1-\frac{\omega_0}{\beta}-C\sigma\right)+\frac{1}{2}\left(F_w^g-\psi_1\right)
\end{equation}

\begin{equation}
\label{10-18}a_3=1+\frac{1}{2\tau}(F_w^g-\psi_1),
\end{equation}

\begin{equation}
\label{40-18} q_2=q+q_1.
\end{equation}

\bigskip

Now we choose
\begin{equation}
\label{50-18} \psi_1=\psi_0+\frac{\varepsilon\sigma^\beta}{\beta},
\end{equation}
where $0<\varepsilon\leq \text{min}\{1,\beta-\omega_0\}$.\\
Let us notice that for every $x\in B_1^{\sigma}\setminus\{0\}$,

\begin{equation}
\label{1-18}C^{-1}\sigma^{-2\beta-2}\leq\frac{|\nabla_gw|^2}{w^2}\leq
C\sigma^{-2\beta-2},
\end{equation}

\begin{equation}
\label{2-18}F_w^g-\psi_1\geq-\frac{\sigma^\beta}{\beta}\left(\varepsilon+C\sigma\right),
\end{equation}

\begin{equation}
\label{3-18}\psi_1-F_w^g\geq\frac{\sigma^\beta}{\beta}\left(\varepsilon-C\sigma\right),
\end{equation}

\begin{equation}
\label{1-19}|\psi_1|\leq C\sigma^{\beta},\qquad
|\Delta_g\psi_1|\leq C\sigma^{\beta-2},
\end{equation}
where $C, C\geq1$, depends only on
$\lambda,\Lambda,m_\ast,m^{\ast}$ and $\beta$, with \eqref{2-18}--\eqref{1-19} following {}from
\eqref{10-16}--\eqref{2*-16} and \eqref{50-18}. {}From \eqref{1-18}--\eqref{1-19}
we have that, for every
$x\in B_1^{\sigma}\setminus\{0\}$ and for every $\tau\geq1$

\begin{equation}
\label{10-19} a_1\geq
C_{\ast}^{-1}\sigma^{-\beta-2}\left(\varepsilon-C_0\sigma-\frac{C_1}{\tau}\sigma^\beta\right),
\end{equation}
where $C_{\ast}, C_0, C_1$ ($C_{\ast}\geq1, C_0\geq1, C_1\geq1$)
depend only on $\lambda,\Lambda,m_\ast,m^{\ast}$ and $\beta$.
Therefore, if $0<\sigma(x)\leq\frac{\varepsilon}{2C_0}$ and
$\tau\geq\frac{4C_1}{\varepsilon}$, then we have

\begin{equation}
\label{2-19} a_1
\geq\frac{\varepsilon}{4}C_{\ast}^{-1}\sigma^{-\beta-2},
\end{equation}
where $C, C\geq1$, depends only on
$\lambda,\Lambda,m_\ast,m^{\ast}$ and
$\beta$.\\
Concerning $a_2$, we have by \eqref{2-18}
\begin{equation}
\label{20-19}
a_2\geq{\sigma}^{\beta}\left(\frac{1}{2}\left(1-\frac{\omega_0}{\beta}\right)-C_2\sigma\right),
\end{equation}
where $C_2, C_2\geq1$, depends only on
$\lambda,\Lambda,m_\ast,m^{\ast}$ and $\beta$. Therefore, if
$0<\sigma(x)\leq\dfrac{\beta-\omega_0}{4\beta C_2}$, then we have
\begin{equation}
\label{3-19}
a_2\geq\frac{1}{4}{\sigma}^{\beta}\left(1-\frac{\omega_0}{\beta}\right),
\end{equation}
Concerning $a_3$, by \eqref{1-19} and \eqref{2*-16} we have that
there exists $C_3, C_3\geq1$, depending only on
$\lambda,\Lambda,m_\ast,m^{\ast}$ and $\beta$ such that if
$\tau\geq C_3$ and $0<\sigma(x)\leq 1$ then
\begin{equation}
\label{4-19} a_3\geq\frac{1}{2}.
\end{equation}

\bigskip

Now, denote by $\tau_0=\text{max}\{\frac{4C_1}{\varepsilon},C_3\}$
and
$r_0=\text{min}\{\frac{\varepsilon}{2C_0},\frac{\beta-\omega_0}{4\beta
C_2}\}$, by \eqref{4*-6}, \eqref{2-17}, \eqref{2-19}, \eqref{3-19}
and \eqref{4-19} we have

\begin{multline}
\label{1-20}
\qquad\frac{w^2}{|\nabla_gw|^2}\left(P_\tau(f)\right)^2\geq
\tau^3\sigma^{-\beta-2}\frac{\varepsilon }{2}C_{\ast}^{-1}f^2+\\
+\tau{\sigma}^{\beta}\left(1-\frac{\omega_0}{\beta}\right)|\nabla_gf|^2+
2\tau F_{w }^{g}f P_{\tau}(f)+ \divrg(q_2),
\end{multline}
for every $x\in B_{r_0}^{\sigma}\setminus\{0\}$ and
$\tau\geq\tau_0$.\\
By Young's inequality, by the first of \eqref{2*-16} and by \eqref{2-18} we have

\begin{equation}
\label{*1-20} |2\tau F_{w }^{g}f
P_{\tau}(f)|\leq\frac{1}{2}\frac{w^2}{|\nabla_gw|^2}\left(P_\tau(f)\right)^2
+C_{4}\tau^2\sigma^{-2}f^2,
\end{equation}
where $C_4, C_4\geq1$, depends only on
$\lambda,\Lambda,m_\ast,m^{\ast}$ and $\beta$.\\
By \eqref{1-20} and \eqref{*1-20} we have

\begin{multline}
\label{2-20}
\qquad\frac{1}{2}\frac{w^2}{|\nabla_gw|^2}\left(P_\tau(f)\right)^2\geq
\tau^3\sigma^{-\beta-2}\frac{\varepsilon }{4}C_{\ast}^{-1}f^2+\\
+\tau{\sigma}^{\beta}\left(1-\frac{\omega_0}{\beta}\right)|\nabla_gf|^2+
 \divrg(q_2),
\end{multline}
for every $x\in B_{r_0}^{\sigma}\setminus\{0\}$ and every
$\tau\geq\tau_1:=\text{max}\{\tau_0,\frac{4C_{\ast}C_4}{\varepsilon}\}$.

\bigskip

Finally, we choose $\varepsilon=\text{min}\{1,\beta-\omega_0\}$.
Recalling that $f=w^{-\tau}u$, and integrating both sides of
\eqref{2-20} over $B_{r_0}^{\sigma}\setminus\{0\}$, we obtain
\eqref{2-15}.
\end{proof}

\begin{rem}
  \label{rem:nondiv} It is straightforward that estimate \eqref{2-15} remains
  valid for operators in non-divergence form
  $Pu=g_{ij}\partial^2_{ij}u$. Of course, the values of the constants, and in particular of $\tau_1$, might
  be different.
\end{rem}
\section{Carleman estimate for product of two second order elliptic operators}\label{Sec4.2}

 {}In this section and in the sequel we return to the standard notation, that is we denote
 by $|\cdot|$ and by $\cdot$ the euclidian  norm and scalar product respectively.

Let $\{g_1^{ij}(x)\}_{i,j=1}^n$ and $\{g_2^{ij}(x)\}_{i,j=1}^n$ be
two symmetric matrix real valued functions which satisfy
conditions \eqref{1-10}, \eqref{2-10} and let us assume that
\begin{equation}
    \label{1-22}
\sum_{i,j=1}^n\|\nabla ^2 g_1^{ij}\|_{L^\infty(\R^n)}\leq
\Lambda_1, \quad\sum_{i,j=1}^n\|\nabla ^2
g_2^{ij}\|_{L^\infty(\R^n)}\leq \Lambda_1,
\end{equation}
with $\Lambda_1>0$. Let us denote by $L_1$, $L_2$ and
$\mathcal{L}$ the operators
\begin{equation}
    \label{10-22}
L_1(u)=\sum_{i,j=1}^n g_1^{ij}(x)\partial_{ij}^2 u,\quad
L_2(u)=\sum_{i,j=1}^n g_2^{ij}(x)\partial_{ij}^2 u,
\end{equation}
\begin{equation}
    \label{1*-22}
   \mathcal{L}(u)=L_2(L_1 u).
\end{equation}
In the sequel we shall need the following standard proposition
which we prove for the reader's convenience.
\begin{prop}
\label{prop:5-4.2} Let $L_1$, $L_2$ and $\mathcal{L}$ be the
operators defined above. Given $a\in C^1(\R^n\setminus\{0\})$ and
$u\in C^\infty_0(\R^n\setminus\{0\})$, the following inequalities
hold true:
\begin{equation}
    \label{2-22}
   \int a^2|\nabla ^2 u|^2\leq C\left(\int a^2|L_k u|^2+\int (a^2 + |\nabla a|^2)|\nabla u|^2\right),\quad k=1,2,
\end{equation}

\begin{equation}
    \label{1-23}
   \int a^2|\nabla ^3 u|^2\leq C\left(\int a^2|\mathcal{L} u||\nabla ^2 u|+\int (a^2 + |\nabla a|^2)|\nabla^2 u|^2\right),
\end{equation}
where $C$ only depends on $\lambda$ and $\Lambda$.
\end{prop}

\begin{proof}
To simplify the notation, let us omit the index $k$ in $L_k$. For a
fixed $l\in\{1,...,n\}$ we have
\begin{multline}
   \int Lu\partial^2_{ll}u a^2=-\int\partial_l(a^2g^{ij}\partial^2_{ij} u)\partial_l u=\\
   =-\int a^2 g^{ij}\partial^3_{ijl} u \partial_l u -2\int a \partial_l a g^{ij}\partial^2_{ij} u \partial_l u-
   \int (\partial_l g^{ij})\partial^2_{ij} u \partial_l u a^2=\\
   =\int a^2 g^{ij}\partial^2_{il} u \partial^2_{jl} u +\int \partial_j(a^2g^{ij})\partial^2_{il} u \partial_l u
   -2\int a \partial_l a g^{ij}\partial^2_{ij} u \partial_l u-
   \int (\partial_l g^{ij})\partial^2_{ij} u \partial_l u a^2\geq\\
   \geq\lambda\int a^2|\nabla\partial_l u|^2-C\int(|a|+|\nabla a|)|a||\nabla u||\nabla^2 u|,
\end{multline}
where $C$ only depends on $\lambda$ and $\Lambda$.

Now, summing up with respect to $l$ the above inequalities and
applying the inequality $2xy\leq x^2+y^2$, we get \eqref{2-22}.

Now we prove \eqref{1-23}. First we observe that, \cite{G-T},
multiplying both sides of the second equality \eqref{10-22} by
$a^2v$ and integrating by parts we easily obtain
\begin{equation}
    \label{2-23}
   \int a^2|\nabla v|^2\leq C\left(\int a^2|L_2 v||v|+\int (a^2 + |\nabla a|^2)v^2\right),
\end{equation}
where $C$ only depends on $\lambda$ and $\Lambda$.

Let us apply \eqref{2-23} to $v=L_1 u$. Noticing that, for a fixed
$l\in\{1,...,n\}$, we have
\begin{equation}
    \label{10-24}
   |L_1(\partial_l u)|\leq |\partial_l (L_1 u)|+ C|\nabla^2 u|,
\end{equation}
where $C$ only depends on $\Lambda$, we obtain
\begin{equation}
    \label{1-24}
   \int a^2|L_1(\partial_l u)|^2\leq C\left(\int a^2|\mathcal{L} u||\nabla^2 u|+\int (a^2 + |\nabla a|^2)|\nabla^2 u|^2\right),
\end{equation}
where $C$ only depends on $\lambda$ and $\Lambda$.

Finally, by applying inequality \eqref{2-22} to estimate {}from
below the integral on the left hand side of \eqref{1-24}, and
summing up with respect to $l$, we get \eqref{1-23}.
\end{proof}
In order to prove the next theorem we need to use some
transformation formulae for the operator $\mathcal{L}$ which we
recall now. Let $\Psi:\R^n\rightarrow \R^n$ be a $C^4$
diffeomorphism. We have
\begin{equation}
    \label{2-24}
     (\mathcal{L} u)(\Psi^{-1}(y))=(\widetilde{\mathcal{L}} U)(y)+(Q U)(y),
\end{equation}
where $U(y)=u(\Psi^{-1}(y))$, $Q$ is a third order operator,
$\widetilde{\mathcal{L}}=\widetilde{L}_2\widetilde{L}_1$,
$\widetilde{L}_k=\sum_{i,j=1}^n
\widetilde{g}_k^{ij}(y)\partial^2_{ij}$, $k=1,2$, and
$\widetilde{g}_k^{-1}(\Psi(x))=\frac{\partial\Psi}{\partial
x}(x)g_k^{-1}(x)\left(\frac{\partial\Psi}{\partial
x}(x)\right)^t$, namely
\begin{equation}
    \label{10-25}
    \widetilde{g}_k^{ij}(\Psi(x))=\sum_{r,s=1}^n g_k^{rs}(x)\frac{\partial\Psi_i}{\partial x_r}(x)\frac{\partial\Psi_j}{\partial x_s}(x),
    \quad i,j=1,...,n.
\end{equation}
We can find  a linear map $\Psi$ such that
$\widetilde{g}_1^{-1}(0)$ is the identity matrix and
$\widetilde{g}_2^{-1}(0)$ is a diagonal matrix. More precisely,
let $R_1$ be the matrix of a rotation such that
$R_1g_1^{-1}(0)R_1^t=\text{diag}\{\nu_1,....\nu_n\}$, where
$\nu_i$, $i=1,...,n$, are the eigenvalues of $g_1^{-1}(0)$, and
let
$H=\text{diag}\{\frac{1}{\sqrt{\nu_1}},...,\frac{1}{\sqrt{\nu_n}}\}$.
We have that $HR_1g_1^{-1}(0)R_1^tH^t$ is  equal to the identity
matrix. Now let $R_2$ be the matrix of a rotation such that
$\widetilde{g}_2^{-1}(0)=R_2HR_1g_2^{-1}(0)R_1^tH^tR_2^t$ has a
diagonal form. We have that the desired map is $\Psi(x)=R_2HR_1x$.
In addition, notice that if $\nu_*$, $\nu^*$ are the minimum and
maximum eigenvalues of $g_1^{-1}(0)$ respectively and $\mu_*$,
$\mu^*$ are the minimum and maximum eigenvalues of $g_2^{-1}(0)$
respectively, then
\begin{equation}
    \label{1-26}
    \frac{\mu_*}{\nu^*}|x|^2\leq \widetilde{g}_2^{-1}(0)x \cdot x \leq \frac{\mu^*}{\nu_*}|x|^2, \quad \hbox{for every }x\in\R^n.
\end{equation}
\begin{theo}
   \label{theo:6-4.2} Let $\mathcal{L}$ be the operator defined by \eqref{1*-22}.
   Let $\nu_*$ and $\nu^*$ ($\mu_*$ and $\mu^*$) be the minimum and the maximum eigenvalues of $g_1^{-1}(0)$ ($g_2^{-1}(0)$).
   Then there exists a symmetric matrix $\Gamma_0$ satisfying
\begin{equation}
    \label{2-26}
    \lambda^2|x|^2\leq \sigma_0^2(x):=\Gamma_0x \cdot x \leq \lambda^{-2}|x|^2,
\end{equation}
and such that if $\beta>\sqrt{\frac{\mu^*\nu^*}{\mu_*\nu_*}}-1$
and
\begin{equation}
    \label{3-26}
    w_0(x)=e^{-\left(\sigma_0(x)\right)^{-\beta}}
\end{equation}
the following inequality holds true:
\begin{equation}
    \label{1-27}
    \sum_{k=0}^3 \tau^{6-2k}\int \sigma_0^{-\beta-2+k(2\beta+2)}w_0^{-2\tau}|\nabla^k u|^2dx\leq
    C\int\sigma_0^{5\beta+6}w_0^{-2\tau}|\mathcal{L} u|^2dx,
\end{equation}
for every $u\in C^\infty_0(B_{r_1}^{\sigma_0}\setminus\{0\})$ and
for every $\tau\geq\overline{\tau}$, where $r_1$, $0<r_1<1$, $C$
and $\overline{\tau}$ only depend on $\lambda$, $\Lambda$ and
$\Lambda_1$.
\end{theo}

\begin{proof} By the comments preceding the statement of the theorem, without loosing of generality we can assume
that $g_1^{ij}(0)=\delta^{ij}$ and $g_2^{-1}(0)$ is of diagonal
form, say $g_2^{-1}(0)=\text{diag}\{\mu_1,\mu_2,...,\mu_n\}$,
where $0<\mu_1\leq \mu_2\leq...\leq\mu_n$. We denote by
$\Gamma=\{\gamma_{ij}\}_{i,j=1}^n$ a symmetric matrix that we
shall choose later on, and by $m_*$ and $m^*$ the minimum and the
maximum eigenvalues of $\Gamma$ respectively, with $m_*>0$. Let us
set $\sigma(x)=(\Gamma x\cdot x)^{1/2}$. We denote by $S_k^{(0)}$,
$k=1,2$, the matrix $S_\sigma^{g_k(0)}$ introduced in
\eqref{*-11}. We denote by $\omega_0^k$ the numbers (compare with
\eqref{2-11})
\begin{equation}
    \label{1-28}
    \omega_0^k=\sup\left\{-(S_k^{(0)}\xi)\cdot\xi\ |\ g_k^{ij}(0)\xi_i\xi_j=1,g_k^{ij}(0)\partial_i\sigma(x)\xi_j=0,
    x\in\R^n\setminus\{0\}\right\}.
\end{equation}
Let $\beta$ be a positive number such that
$\beta>\max\{\omega_0^1,\omega_0^2\}$ and let $V\in
C^\infty_0(B^\sigma_{r_0}\setminus\{0\})$, where $r_0$ has been
defined in Theorem \ref{theo:4-4.1}. Since
\begin{equation}
    \label{10-29}
    |\Delta_{g_k}V|\leq|L_kV|+C|\nabla V|, \quad k=1,2,
\end{equation}
where $C$ only depends on $\Lambda$, by \eqref{2-15} we have that
there exists $\tau_2$, only depending on $\lambda$, $\Lambda$,
$m_*$, $m^*$ and $\beta$ such that for $k=1,2,$ and for every
$\tau\geq \tau_2$
\begin{equation}
    \label{1-29}
    \tau\int\sigma^\beta w^{-2\tau}|\nabla V|^2+\tau^3\int\sigma^{-\beta-2}w^{-2\tau} V^2\leq
     C \int\sigma^{2\beta+2}w^{-2\tau} |L_kV|^2.
\end{equation}
Now we iterate inequality \eqref{1-29}. First we notice that, by a
standard density property, inequality \eqref{1-29} is valid for
every $V\in H^2_0(B^\sigma_{r_0}\setminus\{0\})$. Let $u$ be an
arbitrary function belonging to
$C^\infty_0(B^\sigma_{r_0}\setminus\{0\})$ and let us set $v=L_1
u$. By applying inequality \eqref{1-29} to the function
$V=\sigma^{\frac{3}{2}\beta+2}v$, we get
\begin{multline}
    \label{1-30}
    \tau^3\int\sigma^{2\beta+2}w^{-2\tau}v^2=\tau^3\int\sigma^{-\beta-2}w^{-2\tau} (\sigma^{\frac{3}{2}\beta+2}v)^2\leq\\
     \leq C \int\sigma^{2\beta+2}w^{-2\tau} |L_2(\sigma^{\frac{3}{2}\beta+2}v)|^2,
\end{multline}
for every $\tau\geq\tau_2$.

Now observe that
\begin{equation}
    \label{2-30}
    |L_2(\sigma^{\frac{3}{2}\beta+2}v)|\leq \sigma^{\frac{3}{2}\beta+2}|L_2v|+C\sigma^{\frac{3}{2}\beta+1}|\nabla v|+C\sigma^{\frac{3}{2}\beta}|v|,
\end{equation}
where $C$ only depends on $\lambda$, $\Lambda$, $m_*$, $m^*$ and
$\beta$. By using \eqref{2-30} to estimate {}from above the right
hand side of \eqref{1-30}, we have that there exists
$\tau_3\geq\tau_2$ such that, for every $\tau\geq\tau_3$,
\begin{equation}
    \label{3-30}
    \tau^3\int\sigma^{2\beta+2}w^{-2\tau}v^2\leq
     C \int\sigma^{5\beta+6}w^{-2\tau} |L_2v|^2+C \int\sigma^{5\beta+4}w^{-2\tau} |\nabla v|^2,
\end{equation}
where $C$ and $\tau_3$ only depend on $\lambda$, $\Lambda$, $m_*$,
$m^*$ and $\beta$.

Now we estimate {}from above the second term in the right hand
side of \eqref{3-30}. To this aim we apply inequality \eqref{1-29}
to the function $V=\sigma^{2\beta+2}v$ and we have
\begin{equation}
    \label{1-31}
    \tau\int\sigma^{\beta}w^{-2\tau}|\nabla(\sigma^{2\beta+2}v)|^2\leq
     C \int\sigma^{2\beta+2}w^{-2\tau} |L_2(\sigma^{2\beta+2}v)|^2,
\end{equation}
for every $\tau\geq\tau_2$.

Taking into account that
\begin{equation}
    \label{10-31}
    |L_2(\sigma^{2\beta+2}v)|\leq \sigma^{2\beta+2}|L_2v|+C\sigma^{2\beta+1}|\nabla v|+C\sigma^{2\beta}|v|,
\end{equation}
and
\begin{equation}
    \label{20-31}
    |\nabla(\sigma^{2\beta+2}v)|^2\geq \frac{1}{2}\sigma^{4\beta+4}|\nabla v|^2-C\sigma^{4\beta+2}v^2,
\end{equation}
where $C$ only depends on $\lambda$, $\Lambda$, $m_*$, $m^*$ and
$\beta$, we have, by \eqref{1-31},
\begin{equation}
    \label{2-31}
    \tau\int\sigma^{5\beta+4}w^{-2\tau}|\nabla v|^2\leq
     C \int\sigma^{6\beta+6}w^{-2\tau} |L_2v|^2+C\tau\int\sigma^{5\beta+2}w^{-2\tau}v^2,
\end{equation}
for every $\tau\geq\tau_2$, where $C$ only depends on $\lambda$,
$\Lambda$, $m_*$, $m^*$ and $\beta$.

Now we use \eqref{2-31} to estimate {}from above the second term
on the right hand side of \eqref{3-30} and we have that there
exists $\tau_4\geq\tau_3$ such that
\begin{equation}
    \label{1-32}
    \int\sigma^{2\beta+2}w^{-2\tau}v^2\leq
     \frac{C}{\tau^3} \int\sigma^{5\beta+6}w^{-2\tau} |L_2v|^2,
\end{equation}
for every $\tau\geq\tau_4$, where $C$ and $\tau_4$ only depend on
$\lambda$, $\Lambda$, $m_*$, $m^*$ and $\beta$. Recalling that
$v=L_1u$ and by using \eqref{1-29} for $V=u$ and $k=1$,
\eqref{1-32} yields
\begin{equation}
    \label{2-32}
    \tau^6\int\sigma^{-\beta-2}w^{-2\tau}u^2+ \tau^4\int\sigma^{\beta}w^{-2\tau}|\nabla u|^2\leq
    C \int\sigma^{5\beta+6}w^{-2\tau} |L_2L_1u|^2,
\end{equation}
for every $\tau\geq\tau_4$, where $C$ only depends on $\lambda$,
$\Lambda$, $m_*$, $m^*$ and $\beta$.

Now we prove that
\begin{equation}
    \label{3-32}
    \tau^2\int\sigma^{3\beta+2}w^{-2\tau}|\nabla^2u|^2+\int\sigma^{5\beta+4}w^{-2\tau}|\nabla^3 u|^2\leq
    C \int\sigma^{5\beta+6}w^{-2\tau} |L_2L_1u|^2,
\end{equation}
for every $\tau\geq\tau_4$, where $C$ only depends on $\lambda$,
$\Lambda$, $m_*$, $m^*$ and $\beta$.

Concerning the term with the second order derivatives on the left hand
side of \eqref{3-32}, we can estimate it by using \eqref{2-22}
with $a=(\sigma^{3\beta+2}w^{-2\tau})^{\frac{1}{2}}$ and $k=1$,
obtaining
\begin{equation}
    \label{1-33}
    \int\sigma^{3\beta+2}w^{-2\tau}|\nabla^2u|^2\leq
    C \int\sigma^{3\beta+2}w^{-2\tau} |L_1u|^2+C \tau^2\int\sigma^{\beta}w^{-2\tau} |\nabla u|^2,
\end{equation}
where $C$ only depends on $\lambda$, $\Lambda$, $m_*$, $m^*$ and
$\beta$.

By using \eqref{1-29} for $V=u$ and $k=1$ to estimate {}from above
the second integral on the right hand side of \eqref{1-33} we get
\begin{equation}
    \label{2-33}
    \int\sigma^{3\beta+2}w^{-2\tau}|\nabla^2u|^2\leq
    C\tau \int\sigma^{2\beta+2}w^{-2\tau} |L_1u|^2,
\end{equation}
for every $\tau\geq\tau_2$, where $C$ and $\tau_2$ only depend on
$\lambda$, $\Lambda$, $m_*$, $m^*$ and $\beta$.

Now, by \eqref{1-32} with $v=L_1u$ and by \eqref{2-33}, we have,
for every $\tau\geq\tau_4$,
\begin{equation}
    \label{1-34}
    \tau^2\int\sigma^{3\beta+2}w^{-2\tau}|\nabla^2u|^2\leq
    C \int\sigma^{5\beta+6}w^{-2\tau} |L_2L_1u|^2,
\end{equation}
where $C$ only depends on $\lambda$, $\Lambda$, $m_*$, $m^*$ and
$\beta$.

Now we estimate {}from above the term with the third order
derivatives on the left hand side of \eqref{3-32}. By applying
\eqref{1-23} with $a=(\sigma^{5\beta+4}w^{-2\tau})^{\frac{1}{2}}$,
we have
\begin{equation}
    \label{2-34}
    \int\sigma^{5\beta+4}w^{-2\tau}|\nabla^3u|^2\leq
    C \int\sigma^{5\beta+4}w^{-2\tau} |L_2L_1u||\nabla^2u|+C\tau^2\int\sigma^{3\beta+2}w^{-2\tau} |\nabla^2u|^2,
\end{equation}
where $C$ only depends on $\lambda$, $\Lambda$, $m_*$, $m^*$ and
$\beta$.

Noticing that
\begin{multline}
    \label{10-34}
    \sigma^{5\beta+4}|L_2L_1u||\nabla^2u|=\left(\sigma^{\frac{3}{2}\beta+1}|\nabla^2u|\right)
    \left(\sigma^{\frac{7}{2}\beta+3}|L_2L_1u|\right)\leq\\
    \leq\frac{1}{2}\left(\sigma^{3\beta+2}|\nabla^2u|^2+\sigma^{7\beta+6}|L_2L_1u|^2\right),
\end{multline}
by \eqref{1-34} and \eqref{2-34} we obtain the desired inequality
\eqref{3-32}.

By \eqref{2-32} and \eqref{3-32} we have
\begin{equation}
    \label{1-35}
    \sum_{k=0}^3\tau^{6-2k}\int\sigma^{-\beta-2+k(2\beta+2)}w^{-2\tau}|\nabla^ku|^2\leq
    C\int\sigma^{5\beta+6}w^{-2\tau}|L_2L_1u|^2,
\end{equation}
for every $\tau\geq\tau_4$, where $\tau_4$ and $C$ only depend on
$\lambda$, $\Lambda$, $m_*$, $m^*$ and $\beta$, for every $u\in
C^\infty_0(B^\sigma_{r_0}\setminus\{0\})$.

Now we choose $\Gamma=\Gamma_0:=\text{diag}\{\frac{1}{\sqrt
\mu_1},...,\frac{1}{\sqrt \mu_n}\}$,
$\sigma(x)=\sigma_0(x):=\left(\Gamma_0x \cdot x\right)^{1/2}$,
$w(x)=w_0(x)$, where $w_0(x)$ is defined by \eqref{3-26}. By
Proposition \ref{propRem} we have
$\omega_0^1=\omega_0^2=\sqrt{\frac{\mu_n}{\mu_1}}-1$, hence
estimate \eqref{1-35} holds for
$\beta>\sqrt{\frac{\mu_n}{\mu_1}}-1$. Coming back to the old
variables we obtain \eqref{1-27}.
\end{proof}
\begin{theo}
   \label{theo:7-4.2} Let $\mathcal{L}$ be the operator defined by \eqref{1*-22}. Let $\nu_*$, $\nu^*$,
   $\mu_*$, $\mu^*$ be as defined in Theorem \ref{theo:6-4.2}. Let us assume that $u\in H^4(B_R)$ satisfies
   the inequality
\begin{equation}
    \label{1-39}
    |\mathcal{L}u|\leq N\sum_{k=0}^3 R^{-4+k}|\nabla^k u|, \quad \hbox{in } B_R,
\end{equation}
where $N$ and $R$ are positive numbers. Let
$\beta>\sqrt{\frac{\mu^*\nu^*}{\mu_*\nu_*}}-1$. There exist
positive constants $s_1\in(0,1)$ and $C\geq 1$, $C$ and $s_1$ only
depending on $\lambda$, $\Lambda$, $\Lambda_1$ and $N$ such that,
for every $\rho_1\in(0,s_1R)$ and for every $r$, $\rho$ satisfying
$r<\rho<\frac{\rho_1\lambda^2}{2}$,
\begin{multline}
    \label{2-39}
    \sum_{k=0}^3 \rho^{2k}\int_{B_\rho}|\nabla^k u|^2\leq
    C\max\left\{1,\left(\frac{\rho}{R}\right)^{-(5\beta-2)}\right\}
    e^{C\left((\lambda^{-1}\rho)^{-\beta}-\left(\frac{\rho_1\lambda}{2}\right)^{-\beta}\right)R^\beta}\cdot\\
    \cdot\left(\left(\frac{r}{R}\right)^{5\beta-2}\sum_{k=0}^3 r^{2k}\int_{B_{r}}|\nabla^k u|^2\right)^{\vartheta_0}\cdot
    \left(\left(\frac{\rho_1}{R}\right)^{5\beta-2}\sum_{k=0}^3 \rho_1^{2k}\int_{B_{\rho_1}}|\nabla^k u|^2\right)^{1-\vartheta_0}
    ,
\end{multline}
where
\begin{equation}
    \label{3-39}
    \vartheta_0=\frac{(\lambda^{-1}\rho)^{-\beta}-\left(\frac{\lambda\rho_1}{2}\right)^{-\beta}}
    {\left(\frac{\lambda r}{2}\right)^{-\beta}-\left(\frac{\lambda\rho_1}{2}\right)^{-\beta}}.
\end{equation}
\end{theo}
\begin{proof}
First we observe that, denoting
$\widetilde{g}_k^{-1}(x)=g_k^{-1}(Rx)$,
$\widetilde{L}_k=\widetilde{g}_k^{ij}(x)\partial^2_{ij}$, $k=1,2$,
$\widetilde{\mathcal{L}}=\widetilde{L}_2\widetilde{L}_1$,
$\widetilde{u}(x)=u(Rx)$, $x\in B_1$, inequality \eqref{1-39}
implies
\begin{equation}
    \label{1-40}
    |\widetilde{\mathcal{L}}\widetilde{u}|\leq N\sum_{k=0}^3|\nabla^k \widetilde{u}|, \quad\hbox{in }B_1.
\end{equation}
For simplicity of notation we shall omit the symbol $\tilde{}$ .
Let us introduce the following notation
\begin{equation}
    \label{2-40}
    J(\rho)= \sum_{k=0}^3\rho^{2k}\int_{B_\rho^{\sigma_0}}|\nabla^k u|^2,
\end{equation}
where, we recall, $B_\rho^{\sigma_0}=\{x\in\R^n\ |\
\sigma_0(x)<\rho\}$ and $\sigma_0$ has been defined in Theorem
\ref{theo:6-4.2}. Notice that \eqref{2-26} gives $B_{\lambda
r}\subset B_r^{\sigma_0}\subset B_{\frac{r}{\lambda}}$, for every
$r>0$. In particular inequality \eqref{1-40} is satisfied in
$B_\lambda^{\sigma_0}$. Denote by $R_1=\min\{r_1,\lambda\}$, where
$r_1$ has been introduced in Theorem \ref{theo:6-4.2}. Let
$\rho_1\in(0,R_1]$ and $r\in \left(0,\frac{\rho_1}{2}\right)$. Let
$\eta\in C^4_0(\R)$ such that $0\leq\eta\leq 1$, $\eta\equiv 1$ in
$\left(r,\frac{\rho_1}{2}\right)$, $\eta\equiv 0$ in $\left(0,
\frac{r}{2}\right)\cup(\rho_1,R_1)$,
$\left|\frac{d^k}{dt^k}\eta\right|\leq \frac{C}{r^k}$ in
$[\frac{r}{2},r]$, $\left|\frac{d^k}{dt^k}\eta\right|\leq
\frac{C}{\rho_1^k}$ in $\left[\frac{\rho_1}{2},\rho_1\right]$ for
$k=0,1,...,4$, where $C$ is an absolute constant. In addition, let
$\xi(x)=\eta(\sigma_0(x))$. By a standard density theorem,
inequality \eqref{1-27} holds for the function $\xi(x)u(x)$.

Denote
\begin{equation}
    \label{10-47}
   h_\tau(t)=t^{5\beta-2}e^{\frac{2\tau}{t^\beta}},\quad t\in(0,1).
\end{equation}
By standard calculations, it is simple to derive that there exist
$\overline{\tau}_1\geq\overline{\tau}$, $C$, $s_0\in(0,R_1)$, only
depending on $\lambda$, $\Lambda$, $\Lambda_1$, $\beta$ and $N$,
such that if $\rho_1\leq s_0$, $r<\rho<\frac{\rho_1}{2}$ and
$\tau\geq \overline{\tau}_1$ then
\begin{equation}
    \label{20-47}
   h_\tau(\rho)J(\rho)\leq C h_\tau\left(\frac{r}{2}\right)J(r)+C h_\tau\left(\frac{\rho_1}{2}\right)J(\rho_1).
\end{equation}
Hence
\begin{equation}
    \label{1-47}
   J(\rho)\leq C \left(\left(\frac{r/2}{\rho}\right)^{5\beta-2}
   e^{2\tau\left(-\frac{1}{\rho^\beta}+\frac{1}{(r/2)^\beta}\right)}J(r)+
   \left(\frac{\rho_1/2}{\rho}\right)^{5\beta-2}
   e^{2\tau\left(-\frac{1}{\rho^\beta}+\frac{1}{(\rho_1/2)^\beta}\right)}J(\rho_1)\right),
\end{equation}
for every $\tau\geq\overline{\tau}_1$.

Let us denote
\begin{equation}
    \label{10-50}
    \widetilde{\vartheta}_0=\frac{\rho^{-\beta}-\left(\frac{\rho_1}{2}\right)^{-\beta}}
    {\left(\frac{ r}{2}\right)^{-\beta}-\left(\frac{\rho_1}{2}\right)^{-\beta}},
\end{equation}
\begin{equation}
    \label{20-50}
    \alpha_0=\frac{1}{2}\frac{\log\left(\left(\frac{\rho_1}{r}\right)^{5\beta-2}\frac{J(\rho_1)}{J(r)}\right)}
    {\left(\frac{r}{2}\right)^{-\beta}-\left(\frac{\rho_1}{2}\right)^{-\beta}}.
\end{equation}
If $\alpha_0\geq \overline{\tau}_1$ then we choose $\tau=\alpha_0$
in \eqref{1-47} obtaining
\begin{equation}
    \label{1-51}
   J(\rho)\leq \frac{C}{\rho^{5\beta-2}}\left(r^{5\beta-2}J(r)\right)^{\vartheta_0}
   \left(\rho_1^{5\beta-2}J(\rho_1)\right)^{1-\vartheta_0},
\end{equation}
where $C$ only depends on $\lambda$, $\Lambda$, $\Lambda_1$, $N$
and $\beta$.

If $\alpha_0<\overline{\tau}_1$ then we have trivially
\begin{multline}
    \label{2-51}
   J(\rho)\leq J(\rho_1)=\left(J(\rho_1)\right)^{\vartheta_0}\left(J(\rho_1)\right)^{1-\vartheta_0}\leq\\
   \leq\frac{e^{2\overline{\tau}_1\left(\rho^{-\beta}-\left(\frac{\rho_1}{2}\right)^{-\beta}\right)}}{\rho_1^{5\beta-2}}
   \left(r^{5\beta-2}J(r)\right)^{\vartheta_0}
   \left(\rho_1^{5\beta-2}J(\rho_1)\right)^{1-\vartheta_0}.
\end{multline}
By \eqref{1-51} and \eqref{2-51} and scaling the variables we get \eqref{2-39}.
\end{proof}
\begin{cor}[Unique continuation property]
   \label{cor:8-4.2}
   Let $\mathcal{L}$ be the same operator of Theorem \ref{theo:7-4.2} and let $\nu_*$, $\nu^*$,
   $\mu_*$, $\mu^*$ be as defined in Theorem \ref{theo:6-4.2}. Let us assume that
   $u\in H^4(B_R)$ satisfies the inequality
\begin{equation}
    \label{1-53}
    |\mathcal{L}u|\leq N\sum_{k=0}^3 R^{-4+k}|\nabla^k u|,\quad\hbox{in }B_R,
\end{equation}
where $N$ and $R$ are positive numbers.

Assume that
\begin{equation}
    \label{2-53}
    \int_{B_r} u^2=O\left(e^{-\frac{C_0}{r^\kappa}}\right), \quad \hbox{as } r\rightarrow 0,
\end{equation}
where $C_0>0$ and $\kappa>\sqrt{\frac{\mu^*\nu^*}{\mu_*\nu_*}}-1$.

Then we have
\begin{equation}
    \label{3-53}
    u\equiv 0 \quad \hbox{in } B_R.
\end{equation}
\end{cor}
\begin{proof}
Let us fix $\rho_1\in(0,s_1R)$ and $\rho\in\left(r,\frac{\lambda^2}{2}\rho_1\right)$, where $s_1$ has been
defined in Theorem \ref{theo:7-4.2}. Let
\begin{equation}
    \label{4-53}
    \sqrt{\frac{\mu^*\nu^*}{\mu_*\nu_*}}-1<\beta<\kappa.
\end{equation}
By \eqref{2-39} and by the interpolation inequality
\begin{equation}
    \label{10-54}
    \|u\|_{H^3(B_r)}\leq C \|u\|_{L^2(B_r)}^{\frac{1}{4}}\|u\|_{H^4(B_r)}^{\frac{3}{4}},
\end{equation}
where $C>0$ is an absolute constant, we have
\begin{equation}
    \label{1-54}
    \|u\|^2_{H^3(B_\rho)}\leq C\left(\left(\frac{r}{R}\right)^{5\beta-2}\|u\|_{L^2(B_r)}^{\frac{1}{2}}\right)^{\vartheta_0},
\end{equation}
where $\vartheta_0$ is given by \eqref{3-39} and $C>0$ only
depends on $\lambda$, $\Lambda$, $\Lambda_1$, $N$, $\beta$,
$\rho$, $\rho_1$, $R$ and $\|u\|_{H^4(B_R)}$. By \eqref{2-53} and
\eqref{4-53}, passing to the limit as $r\rightarrow 0$ in
\eqref{1-54}, we obtain $u\equiv 0$ in $B_\rho$. By iteration the
thesis follows.
\end{proof}

\section{Three sphere inequalities for the plate operator\label{Sec4.3}}

In this section we specialize the results of Section \ref{Sec4.2},
in particular we specialize the three sphere inequality proved in
Theorem \ref{theo:7-4.2}, for the plate equation
\begin{equation}
    \label{1-56}
    {\mathcal{L}}u:=\partial_{ij}^2 (C_{ijkl}\partial_{kl}^2 u)=0,
    \quad \hbox{in } B_R,
\end{equation}
where $\{C_{ijkl}(x)\}_{i,j,k,l=1}^{2}$ is a fourth order tensor
that satisfies the hypotheses \eqref{eq:sym-conditions-C-components},
\eqref{eq:3.bound_quantit}, \eqref{eq:3.convex} for $\Omega=B_R$ and the \emph{dichotomy condition}
in $B_R$.

In the following, without loss of generality, we assume $R=1$.

In order to apply Theorem \ref{theo:6-4.2} we need to write the
operator ${\mathcal{L}}$ in the following form
\begin{equation}
    \label{2-57}
    {\mathcal{L}}=L_2 L_1 + \widetilde{Q},
\end{equation}
where $L_1$ and $L_2$ are second order operators which satisfy a
uniform ellipticity condition and whose coefficients belong to
$C^{1,1}(B_1)$ and $\widetilde{Q}$ is a third order operator with
bounded coefficients. In the sequel (Lemma \ref{lem:8-4.3}) we
shall prove that \eqref{2-57} holds true under some additional
assumptions on the tensor $\{C_{ijkl}(x)\}_{i,j,k,l=1}^2$.

Let us denote
\begin{equation}
    \label{2-58}
    p(x;\partial) u = \sum_{h=0}^4 a_{4-h}(x)\partial_1^h
    \partial_2^{4-h}u, \quad \hbox{for every } u\in H^4(B_1),
\end{equation}
where the coefficients $a_i(x)$, $i=0,...,4$, have been defined in \eqref{3.coeff6}, \eqref{3.coeffsmall}.

By \eqref{3.coeff6} we have
\begin{equation}
    \label{3-58}
    {\mathcal{L}}u=p(x;\partial) u + Qu, \quad \hbox{for every } u\in H^4(B_1),
\end{equation}
where $Q$ is a third order operator with bounded coefficients
which satisfies the inequality
\begin{equation}
    \label{4-58}
    |Qu| \leq cM \left ( |\nabla^3 u |+ |\nabla^2 u| \right), \quad \hbox{for every } u\in H^4(B_1),
\end{equation}
and $c$ is an absolute constant. In addition we denote
\begin{equation}
    \label{5-58}
    p(x; \xi) = \sum_{h=0}^4 a_{4-h}(x) \xi_1^h \xi_2^{4-h}, \quad
    x \in \overline{B}_1, \ \xi \in \R^2,
\end{equation}
\begin{equation}
    \label{6-58}
    \widetilde{p}(x; t):= p(x; (t,1))=\sum_{h=0}^4 a_{4-h}(x)t^h, \quad
    x \in \overline{B}_1, \ t \in \R.
\end{equation}
Notice that by \eqref{eq:3.bound_quantit} we have
\begin{equation}
    \label{1-59}
    p(x; \xi) \geq \gamma |\xi|^4, \quad x \in \overline{B}_1, \ \xi \in \R^2,
\end{equation}
\begin{equation}
    \label{2-59}
    \widetilde{p}(x; t) \geq \gamma (t^2+1)^2, \quad x \in \overline{B}_1, \ t \in
    \R.
\end{equation}
Now, for any fixed $x \in \overline{B}_1$, let
$z_k(x)=\alpha_k(x)+i\beta_k(x)$,
$\overline{z}_k(x)=\alpha_k(x)-i\beta_k(x)$ ($k=1,2$) be the
complex solutions to the algebraic equation
$\widetilde{p}(x;z)=0$. Here, $\alpha_k$ and $\beta_k$ are
real-valued functions and $\beta_k(x) >0$, $k=1,2$, for every $x
\in \overline{B}_1$.

We have
\begin{equation}
    \label{3-59}
    p(x; \xi) = p_2(x;\xi)p_1(x;\xi), \quad \hbox{for every } x \in \overline{B}_1, \ \xi \in
    \R^2,
\end{equation}
where
\begin{equation}
    \label{4-59}
    p_k(x; \xi) = g_k^{ij}(x)\xi_i \xi_j, \quad k=1,2, \ x \in \overline{B}_1, \ \xi \in
    \R^2,
\end{equation}
\begin{multline}
    \label{5-59}
    g_k^{11}(x)= \sqrt{a_0(x)}, \ \ g_k^{12}(x)=g_k^{21}(x)=-\alpha_k(x)
    \sqrt{a_0(x)},\\
    g_k^{22}(x)=\sqrt{a_0(x)} ( \alpha_k^2(x)+ \beta_k^2(x)),
    \quad k=1,2, \ \ x \in \overline{B}_1.
\end{multline}
Since in the sequel we have to deal with some basic properties of
polynomials, we recall such properties for what concerns the
polynomial $\widetilde{p}(x; z)$ and we refer the reader to
\cite[Chapter 5]{Wa} for an extended treatment of the issue. For
any fixed $x \in \overline{B}_1$ we denote by $ {\mathcal{D}}(x)$
the absolute value of the discriminant of the polynomial
$\widetilde{p}(x;z)$, that is
\begin{equation}
    \label{1-60}
    {\mathcal{D}}(x)=a_0^6\left (
    (z_1-z_2)(z_1-\overline{z}_1)(z_1-\overline{z}_2)(z_2-\overline{z}_1)(z_2-\overline{z}_2)(\overline{z}_1-\overline{z}_2)
    \right )^2,
\end{equation}
where $a_0=a_0(x)$ and $z_k=z_k(x)=\alpha_k(x)+i\beta_k(x)$,
$k=1,2$. An elementary calculation yields
\begin{equation}
    \label{2-60}
    {\mathcal{D}}(x)=16 a_0^6 \beta_1^2 \beta_2^2
    \left [
    (\alpha_1-\alpha_2)^2+ (\beta_1+\beta_2)^2 \right ]^2 \left [
    (\alpha_1-\alpha_2)^2+ (\beta_1-\beta_2)^2
    \right ]^2.
\end{equation}
In terms of the coefficients $a_h=a_h(x)$, $h=0,1,...,4$, it is
also known that
\begin{equation}
    \label{3-60}
    {\mathcal{D}}(x)= \frac{1}{a_0} |\det S(x)|,
\end{equation}
where $S(x)$ is the $7\times 7$ matrix defined by \eqref {3. S(x)}.

Furthermore, let us denote by $\Psi$ the map of $\R^4$ into $\R^4$
defined by $\Psi(t_1,t_2,w_1,w_2)=
\{\Psi_k(t_1,t_2,w_1,w_2)\}_{k=1}^4$, where
\begin{center}
\( {\displaystyle \left\{
\begin{array}{lr}
    \Psi_1(t_1,t_2,w_1,w_2)=t_1+t_2,
         \vspace{0.12em}\\
    \Psi_1(t_1,t_2,w_1,w_2)= t_1^2+t_2^2+4t_1t_2+w_1+w_2,
        \vspace{0.12em}\\
    \Psi_1(t_1,t_2,w_1,w_2)= t_1(t_2^2 +w_2)+t_2(t_1^2+w_1),
        \vspace{0.12em}\\
    \Psi_1(t_1,t_2,w_1,w_2)=(t_1^2+w_1)(t_2^2+w_2).
        \vspace{0.25em}\\
\end{array}
\right. } \) \vskip -3.0em
\begin{eqnarray}
\ & & \label{5-61}
\end{eqnarray}
\end{center}
Notice that
\begin{equation}
    \label{6-61}
    a_1=-2a_0\Psi_1(\alpha_1, \alpha_2, \beta_1^2, \beta_2^2),
\end{equation}
\begin{equation}
    \label{7-61}
    a_2=a_0\Psi_2(\alpha_1, \alpha_2, \beta_1^2, \beta_2^2),
\end{equation}
\begin{equation}
    \label{8-61}
    a_3=-2a_0\Psi_3(\alpha_1, \alpha_2, \beta_1^2, \beta_2^2),
\end{equation}
\begin{equation}
    \label{9-61}
    a_4=a_0\Psi_4(\alpha_1, \alpha_2, \beta_1^2, \beta_2^2).
\end{equation}
Let us denote by $\frac{\partial \Psi (t_1, t_2, w_1,
w_2)}{\partial (t_1, t_2, w_1, w_2)}$ the jacobian matrix of
$\Psi$ and let $J(t_1, t_2, w_1, w_2)$ be its determinant. An
elementary calculation shows that
\begin{equation}
    \label{1-62}
    J(t_1, t_2, w_1, w_2)= - \left [
    (t_1-t_2)^4+2(w_1+w_2)(t_1-t_2)^2+(w_1-w_2)^2
    \right ].
\end{equation}
Let us denote
\begin{equation}
    \label{3-62}
    \gamma_1= \min \left \{ \gamma, \frac{1}{16M}, 1 \right \}.
\end{equation}

The following lemma holds.

\begin{lem}
   \label{lem:8-4.3}
    Let $p_k(x;\xi)$, $k=1,2$, be defined by \eqref{4-59}. The
    following facts hold:

\noindent (a) If \eqref{eq:sym-conditions-C-components} and \eqref{eq:3.bound_quantit} are satisfied, then

\begin{equation}
    \label{4-62}
    \gamma_2 |\xi|_2^2 \leq p_k(x;\xi) \leq
    \gamma_2^{-1}|\xi|_2^2, \quad \hbox{for every } x\in
    \overline{B}_1, \ \xi \in \R^2, \ k=1,2,
\end{equation}
where $\gamma_2= 5^{-6}\gamma_1^{15}$.

\noindent (b) If the dichotomy condition introduced in Definition \ref{def:dichotomy}
holds true in $B_1$, then
$g_k^{ij} \in C^{1,1}(\overline{B}_1)$, for $i,j,k=1,2$.

More precisely, if \eqref{3.D(x)bound} holds true, then
\begin{equation}
    \label{2-63}
    \sum_{i,j,k=1}^2
    \left (
    \| \nabla g_k^{ij}
    \|_{L^{\infty}(B_1)} \delta_1^{1/2} + \| \nabla^2 g_k^{ij}
    \|_{L^{\infty}(B_1)} \delta_1
    \right )
    \leq C_1,
\end{equation}
where $\delta_1= \min_{\overline{B}_1} {\mathcal{D}}(x)$ and $C_1$
only depends on $M$ and $\gamma$, whereas if \eqref{3.D(x)bound 2} holds true, then
\begin{equation}
    \label{3-63}
    \sum_{i,j,k=1}^2
    \left (
    \| \nabla g_k^{ij}
    \|_{L^{\infty}(B_1)}  + \| \nabla^2 g_k^{ij}
    \|_{L^{\infty}(B_1)}
    \right )
    \leq C_2,
\end{equation}
where $C_2$ only depends on $M$ and $\gamma$.
\end{lem}

\begin{proof}
First we prove (a). Let $x$, $x \in \overline{B}_1$, be fixed. In
the rest of the proof of (a) we shall omit, for brevity, the
dependence on $x$.

By \eqref{1-59}, \eqref{eq:3.bound_quantit}, \eqref{3-62}, we have
\begin{equation}
    \label{1-64}
    \gamma_1 |\xi|^4 \leq  p(\xi) \leq
    \gamma_1^{-1} |\xi|^4, \quad \hbox{for every } \xi \in \R^2.
\end{equation}
Now we observe that the following inequalities hold true
\begin{equation}
    \label{2-64}
    |\alpha_1+\alpha_2| \leq \gamma_1^{-2},
\end{equation}
\begin{equation}
    \label{3-64}
    |\alpha_1^2+\beta_1^2+\alpha_2^2+\beta_2^2+4\alpha_1\alpha_2 | \leq \gamma_1^{-2},
\end{equation}
\begin{equation}
    \label{4-64}
    |\alpha_1(\alpha_2^2+\beta_2^2)+\alpha_2(\alpha_1^2+\beta_1^2) | \leq \gamma_1^{-2},
\end{equation}
\begin{equation}
    \label{5-64}
    \gamma_1^{2}\leq (\alpha_1^2+\beta_1^2)(\alpha_2^2+\beta_2^2) \leq \gamma_1^{-2},
\end{equation}
\begin{equation}
    \label{6-64}
    \gamma_1^{2}(1+\alpha_1^2)^2\leq \beta_1^2\left [ (\alpha_1-\alpha_2)^2+\beta_2^2 \right ] \leq \gamma_1^{-2}(1+\alpha_1^2)^2,
\end{equation}
\begin{equation}
    \label{7-64}
    \gamma_1^{2}(1+\alpha_2^2)^2\leq \beta_2^2\left [ (\alpha_1-\alpha_2)^2+\beta_1^2 \right ] \leq
    \gamma_1^{-2}(1+\alpha_2^2)^2.
\end{equation}
Indeed, by \eqref{1-64} we have
\begin{equation}
    \label{1-65}
    \gamma_1 \leq a_0 \leq \gamma_1^{-1}, \quad \gamma_1 \leq a_4 \leq
    \gamma_1^{-1}.
\end{equation}
On the other hand, by \eqref{1-65} and using \eqref{6-61}, \eqref{7-61},
\eqref{8-61}, \eqref{9-61} we obtain the inequalities
\eqref{2-64}, \eqref{3-64}, \eqref{4-64}, \eqref{5-64},
respectively. Concerning \eqref{6-64}, by using \eqref{1-64} for
$\xi=(\alpha_1,1)$ and taking into account \eqref{3-59}, we have
\begin{equation}
    \label{2-65}
    \gamma_1(1+\alpha_1^2)^2\leq a_0 \beta_1^2\left [ (\alpha_1-\alpha_2)^2+\beta_2^2 \right ] \leq
    \gamma_1^{-1}(1+\alpha_1^2)^2.
\end{equation}
Inequality \eqref{6-64} follows {}from the first of \eqref{1-65}
and \eqref{2-65}. Proceeding similarly for $\xi=(\alpha_2,1)$ we
obtain \eqref{7-64}.

Now, denoting
\begin{equation}
    \label{3-65}
    \epsilon_0 = \frac{\gamma_1^3}{\sqrt{50}},
\end{equation}
we are going to prove that the following inequalities hold
\begin{equation}
    \label{4-65}
    \beta_k > \epsilon_0, \quad k=1,2,
\end{equation}
\begin{equation}
    \label{5-65}
    \beta_k \leq \frac{1}{\gamma_1 \epsilon_0}, \quad k=1,2,
\end{equation}
\begin{equation}
    \label{6-65}
    |\alpha_k| \leq \frac{1}{\gamma_1 \epsilon_0}, \quad k=1,2.
\end{equation}
In order to prove \eqref{4-65}, it is enough to consider the case
$k=1$, as the case $k=2$ can be proved by the same arguments. We
proceed by contradiction and we assume that
\begin{equation}
    \label{1-66}
    \beta_1^2 \leq \epsilon_0^2.
\end{equation}
By \eqref{1-66} and \eqref{6-64} we get
\begin{equation}
    \label{2-66}
    \frac{\gamma_1^{2}}{\epsilon_0^2} \leq
    (\alpha_1-\alpha_2)^2+\beta_2^2,
\end{equation}
hence at least one of the following inequalities must hold
\begin{equation}
    \label{3a-66}
    \frac{\gamma_1^{2}}{2\epsilon_0^2} \leq \beta_2^2,
\end{equation}
\begin{equation}
    \label{3b-66}
    \frac{\gamma_1^{2}}{2\epsilon_0^2} \leq (\alpha_1-\alpha_2)^2.
\end{equation}
If the inequality \eqref{3a-66} holds, then by \eqref{5-64} we
have
\begin{equation}
    \label{4-66}
    \alpha_1^2 \leq \alpha_1^2 +\beta_1^2 \leq
    \frac{\gamma_1^{-2}}{\alpha_2^2+\beta_2^2} \leq
    \frac{\gamma_1^{-2}}{\beta_2^2} \leq 2
    \gamma_1^{-4}\epsilon_0^2,
\end{equation}
hence
\begin{equation}
    \label{1-67}
    |\alpha_1| \leq
    \sqrt{2} \gamma_1^{-2} \epsilon_0,
\end{equation}
and in turn inequalities \eqref{1-67}, \eqref{2-64} imply
\begin{equation}
    \label{2-67}
    |\alpha_2| \leq
    (1+\sqrt{2}\epsilon_0) \gamma_1^{-2}.
\end{equation}
Therefore, by \eqref{3-64}, \eqref{3a-66}, \eqref{1-67},
\eqref{2-67}, and recalling that $\gamma_1 \in (0,1)$, we have
\begin{equation}
    \label{3-67}
    \frac{\gamma_1^2}{2\epsilon_0^2}
    \leq
    \beta_2^2 \leq \alpha_2^2 +\beta_2^2+\alpha_1^2+\beta_1^2 < 25
    \gamma_1^{-4},
\end{equation}
hence we have $\epsilon_0 > \frac{\gamma_1^3}{ \sqrt{50} }$, a
contradiction. Hence, \eqref{3a-66} cannot be true.

If \eqref{3b-66} holds, then we have $|\alpha_1|+|\alpha_2|\geq
|\alpha_1-\alpha_2|\geq \frac{\gamma_1}{\sqrt{2}\epsilon_0}$.
Therefore, at least one of the following inequalities holds
\begin{equation}
    \label{4-67}
    |\alpha_1| \geq \frac{\gamma_1}{2\sqrt{2}\epsilon_0}, \quad
    |\alpha_2| \geq \frac{\gamma_1}{2\sqrt{2}\epsilon_0}.
\end{equation}
If the first of \eqref{4-67} holds, then by \eqref{2-64} we have $|\alpha_2|
\geq |\alpha_1| - \gamma_1^{-2} \geq
\frac{\gamma_1}{2\sqrt{2}\epsilon_0} -\gamma_1^{-2}\geq
\frac{\gamma_1}{4\sqrt{2}\epsilon_0}$ and, analogously, if the second of \eqref{4-67} holds, then we
have $|\alpha_1| \geq \frac{\gamma_1}{4\sqrt{2}\epsilon_0}$.
Hence, if \eqref{3b-66} holds, then we have
\begin{equation}
    \label{1-68}
    |\alpha_1| \geq \frac{\gamma_1}{4\sqrt{2}\epsilon_0}, \quad
    |\alpha_2| \geq \frac{\gamma_1}{4\sqrt{2}\epsilon_0}.
\end{equation}
Inequalities \eqref{1-68} and \eqref{5-64} give
\begin{equation}
    \label{2-68}
    \frac{\gamma_1^2}{32\epsilon_0^2} \leq \alpha_1^2 \leq
    \alpha_1^2 + \beta_1^2 \leq \frac{\gamma_1^{-2}}{\alpha_2^2 +
    \beta_2^2} \leq \frac{\gamma_1^{-2}}{\alpha_2^2} \leq 32
    \gamma_1^{-4}\epsilon_0^2.
\end{equation}
As a consequence of the above inequality we have
$\frac{\gamma_1^3}{32}\leq \epsilon_0^2$, that contradicts
\eqref{3-65}. Therefore, \eqref{1-66} cannot be true and
\eqref{4-65} is proved.

By \eqref{5-64} and \eqref{4-65} we easily obtain \eqref{5-65} and
\eqref{6-65}. Finally, by \eqref{4-65}--\eqref{6-65}, we obtain easily an estimate {}from above and {}from
below of the eigenvalues of the matrices
$\{g_k^{ij}(x)\}_{i,j=1}^2$ {}from which the estimate \eqref{4-62}
follows.

Now we prove the statement (b) of the lemma. By \eqref{1-62},
\eqref{1-65}, \eqref{4-65}--\eqref{6-65} we have
\begin{equation}
    \label{1-69}
    \gamma_3 \sqrt{\mathcal{D}(x)} \leq J(x) \leq \gamma_3^{-1}
    \sqrt{\mathcal{D}(x)}, \quad \hbox{for every } x \in
    \overline{B}_1,
\end{equation}
where
\begin{equation}
    \label{2-69}
    J(x)= |J(\alpha_1(x), \alpha_2(x), \beta_1^2(x),
    \beta_2^2(x))|
\end{equation}
and $\gamma_3= 10^{-6}\gamma_1^{25}\gamma_0^{-3}$.

Assume that \eqref{3.D(x)bound} holds in $B_1$.
In order to prove that $g_{k}^{ij} \in C^{1,1}(\overline{B}_1)$
and to derive estimate \eqref{2-63}, it is enough to apply the
Inverse Mapping Theorem to the map $\Psi$. Indeed, by
\eqref{5-61}, the vector-valued function $\omega(x)=(\alpha_1(x),
\alpha_2(x), \beta_1^2(x), \beta_2^2(x))$ satisfies the following
equality
\begin{equation}
    \label{1-70}
    \Psi(\omega(x))=d(x), \quad x \in
    \overline{B}_1,
\end{equation}
where $d(x)=\left ( - \frac{a_1(x)}{2a_0(x)},
\frac{a_2(x)}{a_0(x)}, -\frac{a_3(x)}{2a_0(x)},
\frac{a_4(x)}{a_0(x)} \right )$, hence by \eqref{eq:3.convex},
\eqref{3.coeff6}, \eqref{3.coeffsmall}, \eqref{1-69}, \eqref{2-69},
\eqref{1-70} we obtain \eqref{2-63}.

If \eqref{3.D(x)bound 2} holds true, then by \eqref{2-60} we have
$\alpha_1(x)=\alpha_2(x)$ and $\beta_1(x)=\beta_2(x)$ for every $x
\in \overline{B}_1$. Therefore, by \eqref{5-61}--\eqref{7-61} we have
\begin{equation}
    \label{2-70}
    \alpha_1(x)=\alpha_2(x)= - \frac{a_1(x)}{4a_0(x)}
\end{equation}
and
\begin{equation}
    \label{3-70}
    \beta_1^2(x)=\beta_2^2(x)= \frac{a_2(x)}{2a_0(x)}-
    \frac{3a_1^2(x)}{16a_0^2(x)}.
\end{equation}
By \eqref{eq:3.convex}, \eqref{3.coeff6}, \eqref{3.coeffsmall}, \eqref{1-65},
\eqref{4-65}, \eqref{2-70} and \eqref{3-70} we get \eqref{3-63}.
\end{proof}
\begin{theo}  [Three sphere inequality - first version]
   \label{theo:9-4.3}
   Let us assume that $u \in H^4({B}_R)$ is a solution to
   the equation
\begin{equation}
    \label{1-71}
    \partial_{ij}^2 (C_{ijkl}(x)\partial_{kl}^2 u)=0,
    \quad \hbox{in } B_R,
\end{equation}
where $\{C_{ijkl}(x)\}_{i,j,k,l=1}^{2}$ is a fourth order tensor
whose entries belong to $C^{1,1}(\overline{B}_R)$. Assume that
\eqref{eq:sym-conditions-C-components},
\eqref{eq:3.bound_quantit}, \eqref{eq:3.convex} and the dichotomy condition are
satisfied in $B_R$.
Let $\gamma_2=5^{-6}\gamma_1^{15}$ and $\beta=
\frac{1}{\gamma_2^2}-1$. There exist positive constants $s_2$,
$0<s_2<1$, and $C$, $C>1$, $s_2$ and $C$ only depending on
$\gamma$, $M$ and on $\delta_1=
\min_{\overline{B}_R} \mathcal{D}$, such that, for every $\rho_1 \in
(0,s_2 R)$ and every $r$, $\rho$ satisfying
$r<\rho<\frac{\rho_1\gamma_2^2}{2}$, the following inequality holds
\begin{multline}
    \label{1-72}
    \sum_{k=0}^3 \rho^{2k} \int_{B_\rho} |\nabla^k u|^2 \leq
    C \exp \left (
    C \left((\gamma_2^{-1}\rho)^{-\beta}-(\gamma_2
    \frac{\rho_1}{2})^{-\beta}\right)R^{\beta}
    \right ) \cdot \\
    \cdot \left (
    \sum_{k=0}^3 r^{2k} \int_{B_r} |\nabla^k u|^2
    \right )^{\theta_1}
    \left (
    \sum_{k=0}^3 \rho_1^{2k} \int_{B_{\rho_1}} |\nabla^k u|^2
    \right )^{1-\theta_1},
\end{multline}
where
\begin{equation}
    \label{2-72}
    \theta_1 = \frac{(\gamma_2^{-1}\rho)^{-\beta}-(\gamma_2
    \frac{\rho_1}{2})^{-\beta}}{(\gamma_2 \frac{r}{2}  )^{-\beta}-(\gamma_2
    \frac{\rho_1}{2})^{-\beta}}.
\end{equation}
\end{theo}
\begin{proof}
Let us define
\begin{equation}
    \label{10-75}
    \widetilde{u}(y)=u(Ry), \quad \widetilde{C}_{ijkl}(y)=C_{ijkl}(Ry), \ \ y \in
    \overline{B}_1, \ \ i,j,k,l=1,2.
\end{equation}
Then, $\widetilde{u} \in H^4(B_1)$ is a solution to the equation
\begin{equation}
    \label{20-75}
    \partial_{ij}^2 (\widetilde{C}_{ijkl}(y)\partial_{kl}^2 \widetilde{u})=0,
    \quad \hbox{in } B_1.
\end{equation}
Now, by Lemma \ref{lem:8-4.3} we have that
\begin{equation}
    \label{3-72}
    {\mathcal{L}}=L_2L_1 \widetilde{u} + Q \widetilde{u},
\end{equation}
where $L_k=p_k(y; \partial)$, $k=1,2$, and
\begin{equation}
    \label{4-72}
    p_k(y; \partial) = g_k^{ij}\partial_{ij}^2, \quad k=1,2.
\end{equation}
Here, $\{g_k^{ij}\}_{i,j=1}^2$, $k=1,2$, satisfy \eqref{2-63} or
\eqref{3-63} (the former whenever \eqref{3.D(x)bound} holds, the
latter whenever \eqref{3.D(x)bound 2} holds),
\begin{equation}
    \label{5-73}
    \gamma_2 |\xi|^2 \leq g_k^{ij}(y)\xi_i\xi_j \leq
    \gamma_2^{-1}|\xi|^2, \quad x \in \overline{B}_1, \ \xi \in
    \R^2,
\end{equation}
and $Q$ is a third order operator with bounded coefficients
satisfying
\begin{equation}
    \label{6-73}
    |Q\widetilde{u}| \leq cM \left (
    |\nabla^3\widetilde{u}| + |\nabla^2\widetilde{u}|
    \right ),
\end{equation}
where $c$ is an absolute constant. Therefore, {}from \eqref{3-72}--\eqref{6-73} and Theorem
\ref{theo:7-4.2}, and coming back to the old variables, we obtain
the three sphere inequality \eqref{1-72}.
\end{proof}
The following Poincar\'{e}-type inequality holds.
\begin{prop} [Poincar\'{e} inequality]
    \label{prop:10-4.3}
There exists a positive constant $C$ only depending on $n$ such
that for every $u\in H^2(B_R, \R^n)$ and for every $r\in(0,R]$
\begin{equation}
    \label{1-76}
    \int_{B_R}|\tilde u_r|^2+R^2\int_{B_R}|\nabla \tilde u_r|^2\leq
   CR^4\left(\frac{R}{r}\right)^n\int_{B_R}|\nabla^2 u|^2,
\end{equation}
where
\begin{equation}
    \label{2-76}
  \widetilde{u}_r(x)=u(x) -(u)_r-(\nabla u)_r \cdot x,
\end{equation}
\begin{equation}
    \label{3-76}
    (u)_r=\frac{1}{|B_r|}\int_{B_r}u, \qquad (\nabla
    u)_r=\frac{1}{|B_r|}\int_{B_r}\nabla u.
  \end{equation}
\end{prop}
\begin{proof}
For a proof we refer to \cite[Example 4.3]{A-M-R4}.
\end{proof}
\begin{prop} [Caccioppoli-type inequality]
    \label{prop:11-4.3}
Let us assume that $u \in H^4(B_R)$ is a solution to the equation
\begin{equation}
    \label{4-76}
    \partial_{ij}^2 (C_{ijkl}(x)\partial_{kl}^2 u)=0,
    \quad \hbox{in } B_R,
\end{equation}
where $\{C_{ijkl}(x)\}_{i,j,k,l=1}^{2}$ is a fourth order tensor
whose entries belong to $C^{1,1}(\overline{B}_R)$. Assume that
\eqref{eq:sym-conditions-C-components}--\eqref{eq:3.convex} are satisfied. We have
\begin{equation}
    \label{5-76}
    \int_{B_{\frac{t}{2}}} |\nabla^3 u|^2 \leq C \int_{B_t}
    \sum_{k=0}^2 \left ( t^{k-3} |\nabla^k u| \right )^2, \quad
    \hbox{for every } t \leq R,
\end{equation}
where $C$ is a positive constant only depending on $\gamma$ and
$M$.
\end{prop}
\begin{proof}
The proof of \eqref{5-76} is essentially the same of the proof of
\cite[Proposition $6.2$]{M-R-V1}. Here, for the reader
convenience, we give a sketch of the proof.

For every $t \in (0,R]$, let $\eta \in C_0^\infty(B_t)$ be such
that $0 \leq \eta \leq1$ in $B_t$, $\eta \equiv 1$ in
$B_{\frac{t}{2}}$ and
\begin{equation}
    \label{1-77}
    \sum_{k=1}^3 t^k |\nabla^k \eta| \leq C, \quad
    \hbox{in } B_t,
\end{equation}
where $C$ is an absolute constant. Multiplying equation
\eqref{4-76} by $\Delta (\eta^6 u)$ and integrating over $B_t$, we
have
\begin{equation}
    \label{2-77}
    \int_{B_t} C_{ijkl} \partial_{kl}^2 u \partial_{ij}^2 \Delta (\eta^6
    u)=0
\end{equation}
and, integrating by parts,
\begin{equation}
    \label{3-77}
    \int_{B_t} \left \{
    C_{ijkl} \partial_{kl}^2 \partial_s u \partial_{ij}^2 \partial_s (\eta^6 u)
    +
    \partial_s(C_{ijkl})\partial_{kl}^2u \partial_{ij}^2 \partial_s (\eta^6 u)
    \right \}=0.
\end{equation}
By \eqref{eq:3.convex}, \eqref{1-77}, \eqref{3-77} and taking into
account that $t\leq R$ we have
\begin{equation}
    \label{1-78}
    \int_{B_t}
    \eta^6
    C_{ijkl} \partial_{kl}^2 \partial_s u \partial_{ij}^2 \partial_s
    u = F[u],
\end{equation}
where $F$ satisfies the inequality
\begin{equation}
    \label{2-78}
    |F[u]| \leq
    CM \int_{B_t} \left (
    \sum_{k=0}^2
    t^{k-3} |\nabla^k u|
    \right )^2 +
    CM\int_{B_t} |\nabla^3 u| \eta^3  \left (
    \sum_{k=0}^2
    t^{k-3} |\nabla^k u|
    \right ),
\end{equation}
where $C$ is an absolute constant. By \eqref{1-78}, \eqref{2-78},
\eqref{eq:3.bound_quantit} and Cauchy inequality ($2ab \leq \epsilon a^2 +
\frac{1}{\epsilon}b^2$, for $\epsilon >0$) we have
\begin{equation}
    \label{3-78}
    \gamma \int_{B_t} \eta^6 |\nabla^3 u|^2 \leq CM^2 \int_{B_t}
    \left (
    \sum_{k=0}^2
    t^{k-3} |\nabla^k u|
    \right )^2.
\end{equation}
Inequality \eqref{5-76} follows immediately by \eqref{3-78}.
\end{proof}
\begin{theo} [Three sphere inequality - second version]
   \label{theo:12-4.3}
   Let $u \in H^4({B}_R)$ be a solution to
   the equation
\begin{equation}
    \label{1-79}
    \partial_{ij}^2 (C_{ijkl}(x)\partial_{kl}^2 u)=0,
    \quad \hbox{in } B_R,
\end{equation}
where $\{C_{ijkl}(x)\}_{i,j,k,l=1}^{2}$ is a fourth order tensor
whose entries belong to $C^{1,1}(\overline{B}_R)$. Assume that
\eqref{eq:sym-conditions-C-components},
\eqref{eq:3.bound_quantit}, \eqref{eq:3.convex} and the dichotomy condition are
satisfied in $B_R$. Let $\gamma_2=5^{-6}\gamma_1^{15}$ and $\beta=
\frac{1}{\gamma_2^2}-1$. There exist positive constants $s$,
$0<s<1$, and $C$, $C\geq 1$, $s$ and $C$ only depending on
$\gamma$, $M$ and on $\delta_1=
\min_{\overline{B}_R} \mathcal{D}$, such that, for every $\rho_1 \in
(0,s R)$ and every $r$, $\rho$ satisfying
$r<\rho<\frac{\rho_1\gamma_2^2}{2}$, the following inequality holds
\begin{multline}
    \label{1-80}
    \rho^4 \int_{B_\rho} |\nabla^2 u|^2 \leq
    C \exp \left(
    C \left((\gamma_2^{-1}\rho)^{-\beta}-(\gamma_2
    \frac{\rho_1}{2})^{-\beta}\right)R^{\beta}
    \right) \cdot \\
    \cdot \left (
    r^{4} \int_{B_{2r}} |\nabla^2 u|^2
    \right )^{\theta_1}
    \left (
    \frac{\rho_1^6}{r^2} \int_{B_{2\rho_1}} |\nabla^2 u|^2
    \right )^{1-\theta_1},
\end{multline}
where
\begin{equation}
    \label{2-80}
    \theta_1 = \frac{(\gamma_2^{-1}\rho)^{-\beta}-(\gamma_2
    \frac{\rho_1}{2})^{-\beta}}{(\gamma_2 \frac{r}{2}  )^{-\beta}-(\gamma_2
    \frac{\rho_1}{2})^{-\beta}}.
\end{equation}
\end{theo}
\begin{proof}
Let $a \in \R$, $\omega \in \R^2$ to be chosen later on. Since $u$
is a solution to \eqref{1-79}, also $v=u-a-\omega\cdot x$ is a
solution to \eqref{1-79}. By \eqref{1-72} we have
\begin{equation}
    \label{3-80}
    \rho^4 \int_{B_\rho} |\nabla^2 v|^2 \leq K \left ( H_v(r)
    \right )^{\theta_1} \left ( H_v(\rho_1)
    \right )^{1-\theta_1},
\end{equation}
where
\begin{equation}
    \label{10-81}
    K=C \exp \left (
    C \left((\gamma_2^{-1}\rho)^{-\beta}-(\gamma_2
    \frac{\rho_1}{2})^{-\beta}\right)R^{\beta}
    \right )
\end{equation}
and
\begin{equation}
    \label{20-81}
    H_v(t)= \sum_{k=0}^3 t^{2k} \int_{B_t} |\nabla^k v|^2, \quad
    t \in (0,R).
\end{equation}
By Proposition \ref{prop:11-4.3} we have
\begin{equation}
    \label{1-81}
    H_v(r)= C\sum_{k=0}^2 r^{2k} \int_{B_{2r}} |\nabla^k v|^2,
\end{equation}
where $C$ only depends on $M$ and $\gamma$. Now, we choose
\begin{equation}
    \label{30-81}
    a = \frac{1}{|B_{2r}|}\int_{B_{2r}} u, \quad \omega= \frac{1}{|B_{2r}|}\int_{B_{2r}} \nabla
    u.
\end{equation}
By Proposition \ref{prop:10-4.3} and {}from \eqref{1-81} we have
\begin{equation}
    \label{2-81}
    H_v(r) \leq Cr^4 \int_{B_{2r}} |\nabla^2 u|^2,
\end{equation}
where $C$ only depends on $M$ and $\gamma$.

Similarly, by applying Propositions \ref{prop:10-4.3} and
\ref{prop:11-4.3} we obtain
\begin{equation}
    \label{3-81}
    H_v(\rho_1) \leq C \rho_1^4  \left ( \frac{\rho_1}{r}  \right )^2 \int_{B_{2\rho_1}} |\nabla^2 u|^2,
\end{equation}
where $C$ only depends on $\gamma$ and $M$. {}From \eqref{3-80},
\eqref{1-81}, \eqref{2-81}, inequality \eqref{1-80} follows.
\end{proof}

\begin{theo} [Three sphere inequality - third version]
   \label{theo:13-4.3}
   Let $u \in H^4({B}_R)$ be a solution to
   the equation
\begin{equation}
    \label{1-79bis}
    \partial_{ij}^2 (C_{ijkl}(x)\partial_{kl}^2 u)=0,
    \quad \hbox{in } B_R,
\end{equation}
where $\{C_{ijkl}(x)\}_{i,j,k,l=1}^{2}$ is a fourth order tensor
whose entries belong to $C^{1,1}(\overline{B}_R)$. Assume that
\eqref{eq:sym-conditions-C-components},
\eqref{eq:3.bound_quantit}, \eqref{eq:3.convex} and the dichotomy condition are
satisfied in $B_R$.
Let $\gamma_2=5^{-6}\gamma_1^{15}$ and $\beta=
\frac{1}{\gamma_2^2}-1$. There exist positive constants $s$,
$0<s<1$, and $C$, $C\geq 1$, $s$ and $C$ only depending on
$\gamma$, $M$ and on $\delta_1=
\min_{\overline{B}_R} \mathcal{D}$, such that, for every $\rho_1 \in
(0,s R)$ and every $r$, $\rho$ satisfying
$r<\rho<\frac{\rho_1\gamma_2^2}{2}$, the following inequality holds
\begin{multline}
    \label{3sfere_Cauchy}
    \int_{B_\rho} u^2 \leq
    C \exp \left (
    C ((\gamma_2^{-1}\rho)^{-\beta}-(\gamma_2
    \frac{\rho_1}{2})^{-\beta})R^{\beta}
    \right ) \cdot \\
    \cdot \left (
    \int_{B_{r}} u^2
    \right )^{\theta}
    \left (
    \sum_{k=0}^4 \rho_1^{2k} \int_{B_{\rho_1}} |\nabla^k u|^2
    \right )^{1-\theta},
\end{multline}
where $\theta=\frac{\theta_1}{4}$, with $\theta_1$ given by
\eqref{2-72}

\end{theo}
\begin{proof}

It follows immediately {}from \eqref{1-72} and by the
interpolation inequality
\begin{equation*}
   \|u\|_{H^3(B_r)}\leq C \|u\|_{L^2(B_r)}^{\frac{1}{4}}\|u\|_{H^4(B_r)}^{\frac{3}{4}},
\end{equation*}
where $C$ is an absolute constant and the norms are normalized
according to the convention made in Section \ref{SecCauchy}.
\end{proof}

\textit{Acknowledgements}. We wish to express our gratitude to
Professor Luis Escauriaza for deep, fruitful and stimulating
discussions on the issues of Carleman estimates.

\end{document}